\documentclass[a4paper,11pt]{article}
\usepackage[lmargin=2.5cm,rmargin=2.5cm,tmargin=4.0cm,bmargin=4.0cm]{geometry}
\usepackage[toc]{appendix}
\usepackage{bm}
\usepackage{xspace}
\usepackage{setspace}
\makeatletter

\usepackage{lmodern}

\usepackage[english,french]{babel}
\usepackage[latin1]{inputenc}
\usepackage[T1]{fontenc}
\usepackage{graphicx}
\usepackage{multicol}
\usepackage{float}
\usepackage{color}
\usepackage{fancyhdr}
\usepackage{geometry}%
\usepackage{amsmath}
\usepackage{amssymb}
\usepackage{amsthm}
\usepackage{stmaryrd}
\usepackage{pstricks}
\usepackage{listings}
\usepackage{subfig}
\usepackage{dsfont}
\usepackage{enumerate}
\usepackage{booktabs}
\usepackage[round]{natbib}

\newtheorem{theo}{Theorem}
\newtheorem{lem}{Lemma}
\newtheorem{propo}{Proposition}

\theoremstyle{definition}
\newtheorem{rem}{Remark}

\numberwithin{equation}{section}
\numberwithin{theo}{section}
\numberwithin{lem}{section}
\numberwithin{coro}{section}
\numberwithin{exemple}{section}
\numberwithin{rem}{section}
\numberwithin{propo}{section}
\numberwithin{definition}{section}
\numberwithin{theoFr}{section}

\numberwithin{coroFr}{section}

\numberwithin{definitionFr}{section}

\newcommand{\E}{E}
\newcommand{\V}{\mbox{Var}}
\newcommand{\Cov}{\mbox{Cov}}
\newcommand{\VaR}{\mbox{VaR}}
\newcommand{\PP}{\mathbb{P}}
\newcommand{\arrowloi}{\stackrel{\mathcal{L}}{\rightarrow}}
\newcommand{\eps}{\varepsilon}

\RequirePackage[colorlinks=true,linkcolor=black,citecolor=blue,urlcolor=blue]{hyperref}

\usepackage{microtype}
\begin{document}
\selectlanguage{english}

\title{{
Maximum Likelihood Estimation for Conditionally Heteroscedastic
Models when the Innovation Process is in the Domain of Attraction of
a Stable Law }}
\author{
{\sc Guillaume Lepage\footnote{CREST, 15 boulevard Gabriel Péri,
92245 Malakoff Cedex, France. E-mail: guillaume.lepage@ensae.fr}}}

\date{} \maketitle
\begin{quote}
\begin{center}
%(VERY PRELIMINARY)

\vspace{10pt}

\textbf{Abstract}
\end{center}
{\small \hspace{1em} We prove the strong consistency and the
asymptotic normality of the maximum likelihood estimator of the
parameters of a general conditionally heteroscedastic model with
$\alpha$-stable innovations. Then, we relax the assumptions and only
suppose that the innovation process converges in distribution toward
a stable process. Using a pseudo maximum likelihood estimator with a
stable density, we also obtain the strong consistency and the
asymptotic normality of the estimator. This framework seems relevant
for financial data exhibiting heavy tails. We apply this method to
several financial index and compute stable Value-at-Risk.}
\end{quote}


\medskip


\newpage

\section{Introduction}

ARCH models, introduced by \citet{engle1982ach} and generalized by
\citet{bollerslev1986generalized} are some of the most popular
models for explaining financial time series. In these models, the
time series is stationary but possesses a time varying conditional
variance, this property can be used to explain some of the stylized
facts that can be found in financial series. The GARCH modeling
explains the volatility clustering but it also explains a fraction
of the leptokurticity that can be found in financial time series.
Empirical evidences can be found in the survey article by
\citet{shephard1996statistical}. The most widely used estimator for
the parameters of the GARCH model is the Gaussian Quasi Maximum
Likelihood Estimator (QMLE). To implement this estimator, the
Gaussian density is used to compute the likelihood of the model,
even if the exact distribution of the error process remains
unspecified. Under appropriate assumptions, the Gaussian QMLE is
Consistent and Asymptotically Normal (CAN), see
\citet{berkes2003gps} or \citet{FZ2004mle}.

Most of the assumptions required for the Gaussian QMLE to be CAN are
mild, since one does not need to specify the true distribution of
the error process, the model is less risky to be misspecified as in
the Maximum Likelihood Estimation (MLE) case. The only assumption
that can be challenged is the requirement that the error process
possesses a finite fourth moment. The GARCH model and its
derivatives are mostly applied to financial data which are known to
be heavy tailed. \citet{mandelbrot1963variation} and
\citet{fama1965behavior} found that the unconditional distributions
of most financial returns are heavy tailed and therefore do not
necessarily possess a finite fourth moment. Now even if the GARCH
modeling explains a part of the leptokurticity of the financial time
series, the residuals are often found to remain heavy tailed. For
this reason, there were several attempts to use GARCH models with
non-Gaussian innovation, see \citet{berkes2004efficiency} for a
general approach. GARCH models with heavy tailed distributions have
been studied, \citet{bollerslev1987conditionally} use the student
\emph{t} distribution and \citet{liu1995maximum} used an
$\alpha$-stable distribution for the error process and studied the
model empirically, see also \citet{mittnik2003prediction},
\citet{embrechts1997modelling}.

In this paper we study a stable Maximum Likelihood Estimator (MLE) of a general conditionally heteroskedastic model in which the errors follow a stable distribution. To the best knowledge of the author, the CAN
property of the MLE of such a model with stable innovation has not
been proven, even in the GARCH case where the model was only studied
empirically. Here we prove such a result under a few assumptions about the functional form of the volatility process. By specifying the distribution of the error process $(\eta_t)_t$ to be $\alpha$-stable, we obtain a less general method than the Gaussian QMLE but we do not need any moment assumption and the model takes
into account the fact the data can be heavy tailed.

The Gaussian QMLE possesses the robustness property that even if the errors are not Gaussian, provided that their distribution is in the domain of attraction of
the Gaussian law, the QML estimator is still CAN. We want to obtain
a similar property for the stable GARCH model. Since the only
probability distributions to possess a domain of attraction are the
Gaussian distribution and the family of stable laws, we use this
fact to obtain a robustness property for the stable estimation. In other words, we study the asymptotic behavior of the MLE written for stable innovations when the error process is not stable but close to a stable distribution. With the Generalized
Central Limit Theorem (GCLT) (see \citet{gnedenko1968limit}), we can
characterize the domain of attraction of a stable law. A sum of
i.i.d random variables with certain properties will converge in
distribution to a stable variable. If the innovation process can be
written as a sum of variables, then if the sum converges, it
converges in distribution toward a stable law. We use this property
to give a more general result than the stable MLE. We prove that if
the innovation process is not stable but converges in distribution
to a stable variable, then the stable MLE (which in this case is a
pseudo MLE) is still CAN.

We will study a general class a conditionally heteroscedastic model,
defined by
\begin{equation}
    \label{eq: intro, equation du modèle}
    \left\{
        \begin{aligned}
            \epsilon_t &= \sigma_t \eta_t\\
            \sigma_t &= g\left(\epsilon_{t-1},\epsilon_{t-2},\ldots;\theta_0\right),
        \end{aligned}
    \right.
\end{equation}
where $(\epsilon_t)_t$ is the observed process ( $\epsilon_t\in
\mathbb{R}$), $(\eta_t)_t$ is a sequence of independent and
identically (i.i.d) random variables (the error process), $\theta_0$
is a parameter belonging to a parameter space $\Theta$ and $g:\
\mathbb{R}^{\infty}\times\Theta \mapsto \mathbb{R}_+^*$. This model
contains most of the numerous derivatives of the GARCH model that have been
introduced such as EGARCH, TGARCH and many others, see
\citet{bollerslev2008glossary} for a exhaustive (at the time) list.
Model~\eqref{eq: intro, equation du modèle} contains the classical
GARCH$(p,q)$ model given by
\begin{equation}
    \label{eq: intro, equation du modèle GARCH(p,q)}
    \sigma_t^2 = \omega + \sum_{i=1}^q a_i \epsilon_{t-i}^2
        +\sum_{j=1}^p b_j \sigma_{t-j}^2.
\end{equation}

The plan of this paper is as follows. We recall useful results concerning the stable distribution in Section 2. In the third section, we study a conditional heteroscedastic model with stable innovation and
prove that the MLE is stable. In Section 4, we consider the
case where the stable density is used to compute a pseudo MLE when
the error process is not stable but converges in distribution toward
a stable process. Then, we present in Section 5 some simulation results and
some financial applications.

\section{Properties of stable distributions}
\label{sec: stable}

Since the pioneer work of Mandelbrot, the class of stable
distributions is commonly used in finance and in other areas such as
engineering, signal processing and many other. There are
empirical evidences that some financial processes, denoted $(X_t)_t$,
have regularly varying (heavy) tails, that is, $\PP\left[X_t
>x\right] \sim K x^{-\alpha}$, when $|x|\rightarrow +\infty$, where
$K$ is a constant and $\alpha\in\left(0,2\right)$. Such a process
has infinite variance, therefore the standard Central Limit Theorem (CLT)
cannot be applied. Fortunately, the CLT can be generalized. An iid
random process $(X_t)_t$ with regularly varying tails with index
$\alpha<2$ is in the domain of attraction of a stable law, i.e.
that there exist sequences $(a_n)_n$ and $(b_n)_n$ such that
\begin{equation*}
    \frac{X_1+\cdots+X_n}{a_n}-b_n \arrowloi Y,\ \mbox{where}
    \ Y\ \mbox{is a stable law with tail parameter $\alpha$}.
\end{equation*}
Only stable distributions possess a domain of attraction.
Obviously, the Gaussian law is a stable distribution since the CLT
states that every random variable with finite variance is in the
domain of attraction of the Gaussian law. The Gaussian distribution is
a particular case of a stable distribution with $\alpha=2$.

The formal definition of a stable variable is quite simple: non
degenerate iid random variables $(Z_t)_t$ are stable if there exist
$a_n>0$ and $b_n$ such that $\frac{Z_1+\cdots+Z_n}{a_n}-b_n
\stackrel{\mathcal{L}}{=} Z_1$. For a stable law, there exists, in
general, no closed form of the probability density function. A stable
variable is characterized by four parameters, the previously
mentioned tail exponent $\alpha$, a parameter of asymmetry
$|\beta|\leq 1$, a location parameter $\mu\in\mathbb{R}$ and a scale
parameter $\gamma>0$. When $\beta=0$, the distribution is symmetric
about $\mu$. There are several special cases apart from the Gaussian case with $\alpha=2$, where the density is explicit. A stable distribution with $\alpha=1$ and $\beta=0$ is a
Cauchy distribution. When $\alpha=1/2$ and $\beta=1$, we obtain a Lévy distribution. 

Though the density of a stable variable cannot be written in closed
form in the general case, we can write its characteristic function,
in the (M) parametrization of \citet{zolotarev1986one}. The
random variable $X$ is called stable with parameter
$\psi=(\alpha,\beta,\mu,\gamma)$ (we write $X\sim
S(\psi)=S(\alpha,\beta,\mu,\gamma)$) if
\begin{equation}
    \label{eq: stable, characteristic}
    \varphi_\psi(t) = \E\left[\exp itX\right] =
    \left\{ \begin{aligned}
        &\exp\left\{-\left|\gamma t\right|^{\alpha}
            +\gamma ^\alpha i \beta\tan\left(\frac{\alpha\pi}{2}\right)
            t\left(|t|^{\alpha-1}-1\right) + i\mu t\right\}
            \ &\text{if}\ \alpha \neq 1,\\
        &\exp\left\{-\left|\gamma t\right| - \gamma i \beta t \frac{2}{\pi} \log|\gamma t| 
            +i\mu t\right\}\ &\text{if}\ \alpha = 1.\\
    \end{aligned}\right.
\end{equation}
For other parameterizations or more properties on stable distributions, see
\citet{zolotarev1986one} or \citet{samorodnitsky1994stable}. The
parameterization in~\eqref{eq: stable, characteristic} possesses the advantage of being continuous and differentiable with respect to all parameters, even for
$\alpha=1$. Using the inverse Fourier transform, we can express the
density $f(.,\psi)$ with the characteristic function
\begin{equation}
    \label{eq: stable, inverse Fourier}
    f(x,\alpha,\beta,\mu,\gamma) =
    f(x,\psi) = \frac{1}{2\pi}\int_{\mathbb{R}}e^{-itx}\varphi_\psi(t)dt.
\end{equation}
From \citet{bergstrom1952some}, we give a series expansion of the
stable density which will be useful to easily obtain properties of
the stable distribution and to numerically compute the stable
density.

\begin{propo}
    \label{prop: développement en série}
    For $\alpha <1$, we have
    \begin{eqnarray}
        f(x,\alpha,\beta,0,1) &=& \frac{1}{\pi}\sum_{k\geq 1}(-1)^{k+1}
        \frac{\Gamma(k\alpha+1)}{k!} \frac{\left(1+\tau^2\right)^{k/2}}{\left|x+\tau\right|^{k\alpha+1}}\nonumber\\
        &&\times \sin\left[k\left(\arctan\tau+\frac{\alpha\pi}{2}\right)
        +(k\alpha+1)\pi\mathds{1}_{x<-\tau}\right],
        \label{eq: stable, développement en série alpha < 1}
    \end{eqnarray}
    with $\tau=\beta\tan\frac{\alpha\pi}{2}$. For $\alpha>1$,
    the latter series does not converge but the partial sum of
   ~\eqref{eq: stable, développement en série alpha < 1}
    provides an asymptotic expansion when $|x|$ tends to infinity, i.e. the remainder term has smaller order of magnitude (for large $|x|$) than the last term in the partial sum.
    In the case $\alpha>1$, we have another convergent series expansion
    given by
    \begin{equation}
        \label{eq: stable, développement asymptotique alpha >1}
        f(x,\alpha,\beta,0,1) = \frac{1}{\alpha\pi}\sum_{k\geq 0}
        \frac{\Gamma\left(\frac{k+1}{\alpha}\right)}{k!}x^k
        \cos\left[\frac{k+1}{\alpha}\arctan\tau-\frac{k\pi}{2}\right].
    \end{equation}
\end{propo}

%RETROUVER LA PREUVE POUR INSERER DANS LA THESE

This proposition can be proven as in Bergström, the parameterization
differs but the idea is the same. These series expansions will be
used to numerically compute the stable density. Depending on the
parameters $\alpha$ and $\beta$ and on the value of $x$, the series
\eqref{eq: stable, développement en série alpha < 1} or~\eqref{eq:
stable, développement asymptotique alpha >1} will efficiently
approximate the density $f$. If these series do not provide a good estimation, we can use the Fast-Fourier Transform (FFT) or the Laguerre quadrature, see
\citet{nolan1997numerical} or \citet{matsui2006some}.

In the last proposition, the parameters $\mu$ and $\gamma$ were
fixed to 0 and 1. To obtain a more general formula we can use the
following relation
\begin{equation}
    \label{eq: stable, relation param densité}
    f(x,\alpha,\beta,\mu,\gamma) = \frac{1}{\gamma}
    f\left(\frac{x-\mu}{\gamma},\alpha,\beta,0,1\right).
\end{equation}

From~\eqref{eq: stable, inverse Fourier},~\eqref{eq: stable,
développement en série alpha < 1} and~\eqref{eq: stable, relation
param densité}, it follows that the density of a stable random
variable is infinitely often differentiable with respect to $x$, $\alpha$,
$\beta$ and $\mu$. From the asymptotic expansion~\eqref{eq:
stable, développement en série alpha < 1}, we have the tail behavior
of the density $f$ and all its derivatives. When $x\rightarrow
\pm\infty$,
\begin{align}
    f(x,\psi) &\sim K |x|^{-\alpha-1}, \label{eq:tail behavior 1}\\
    f'(x,\psi) = \frac{\partial f}{\partial x}(x,\psi) &\sim K |x|^{-\alpha-2},
        \label{eq:tail behavior 2}\\
    \frac{\partial f}{\partial \alpha}(x,\psi)&\sim K \log(|x|)|x|^{-\alpha-1},
    \label{eq:tail behavior 3}\\
    \frac{\partial f}{\partial \beta}(x,\psi)&\sim K |x|^{-\alpha-1},
    \label{eq:tail behavior 4}\\
    \frac{\partial f}{\partial \mu}(x,\psi) &\sim K |x|^{-\alpha-2},
    \label{eq:tail behavior 5}
\end{align}
where $K$ is a generic constant, which is not necessarily the same
depending on whether $x\rightarrow+\infty$ or $x\rightarrow-\infty$.

The idea of using stable laws comes from the fact that only a stable
variable possesses a domain of attraction. The Gaussian distribution
is a particular case of stable distribution (with $\alpha=2$), its
domain of attraction contains all distributions with finite
variance. The following is a  CLT for heavy tailed
regularly varying distributions in the particular case of a variable in the domain of normal attraction of a stable law.

\begin{theo}[\citet{gnedenko1968limit}, Theorem 5, $\S$ 35]
    \label{thm: stable, GTCL}
    If the process $(X_t)_t$ is iid with
    \begin{align}
        \PP\left[X_t >x\right] &\sim K_1 x^{-\alpha}\ \mbox{when}\ x\rightarrow\ +\infty,
            \label{eq: stable, conditions domaine d'attraction 1}\\
        \PP\left[X_t <x\right] &\sim K_2 |x|^{-\alpha}\ \mbox{when}\ x\rightarrow\ -\infty,
            \label{eq: stable, conditions domaine d'attraction 2}
    \end{align}
    with $\alpha\in(0,2)$, $K_1>0$ and $K_2>0$, then, with
    \begin{align*}
        \beta &= \frac{K_1-K_2}{K_1+K_2},\
        a = \left\{-\alpha M(\alpha)(K_1+K_2)\cos\frac{\alpha\pi}{2}\right\}^{1/\alpha},\\
        M(\alpha) &= \left\{
            \begin{aligned}
                -\frac{\Gamma(1-\alpha)}{\alpha},\ &\mbox{when}\ \alpha <1,\\
                \frac{\Gamma(2-\alpha)}{\alpha(\alpha-1)},\ &\mbox{when}\ \alpha >1,
            \end{aligned}
            \right.
    \end{align*}
    we have
    \begin{equation}
        \frac{1}{an^{1/\alpha}}\sum_{t=1}^n \left(X_t-m\right)
            \arrowloi Z, \label{eq: stable, result GTCL}
    \end{equation}
    where $m=\E X_1$ if $\alpha>1$, $m=0$ if $\alpha<1$ and $Z\sim
    S(\alpha,\beta,\beta\tan\frac{\alpha\pi}{2},1)$.
\end{theo}

%For our results in section~\ref{sec: pseudo GARCH}, we will need
%some convergences in expectation, therefore the convergence in
%distribution is not sufficient, we will need a convergence in
%densities, for that we give 
%
The following theorem, due to \citet{gnedenko1968limit} (Theorem 2, $\S$ 46), gives a uniform version of the previous result.
\begin{theo}
    \label{thm: stable, GTCL for densities}
    Under the assumptions and with the notations of Theorem~\ref{thm:
    stable, GTCL}, if $X_t$ has a density and if this density is of
    bounded variation, we have
    \begin{equation}
        \label{eq: stable, convergence en densité}
        \underset{x\in\mathbb{R}}{\sup}\ \left|f_n(x)-f(x,\psi)\right|
            \rightarrow 0,\quad \mbox{when}\ n\rightarrow+\infty,
    \end{equation}
    where $f_n$ is the density of $\frac{1}{an^{1/\alpha}}\sum_{t=1}^n
    \left(X_t-m\right)$ and $f(.,\psi)$ is the probability density of $Z$ with $\psi = \left(\alpha,\beta,\beta\tan \frac{\alpha\pi}{2},1\right)$ as defined in Theorem~\ref{thm: stable, GTCL}.
\end{theo}

The previous theorem has been extended by \citet{basu1980local} as follows.
\begin{theo}
    \label{thm, stable basu}
    Under the assumptions and with the notations of Theorem~\ref{thm:
    stable, GTCL for densities}, if the characteristic function $w$ of
    $X_1$ is such that
    \begin{equation*}
        \int_\mathbb{R} \left|w(u)\right|^r du < \infty,
    \end{equation*}
    for some integer $r\geq 1$, then for $0\leq \delta\leq\alpha$, we
    have
    \begin{equation}
        \label{eq: stable, final convergence}
        \underset{x\in\mathbb{R}}{\sup}\ (1+|x|)^\delta
            \left|f_n(x)-f(x,\psi)\right| \rightarrow 0,\quad
            \mbox{when}\ n\rightarrow+\infty.
    \end{equation}
\end{theo}

\section{Conditionally heteroscedastic model with stable
innovations}\label{sec: stable garch}

In this section, we study the properties of the ML estimator, for the general class of conditionally heteroscedastic models defined in~\eqref{eq: intro, equation du
modèle} with a stable error process. The probability distribution of
$(\eta_t)_t$ is a stable law with parameter $\psi=
(\alpha,\beta,\mu,1)$. For identifiability reasons, the parameter
$\gamma$ has to be fixed to 1.

Since we work with a general model, we make some general assumptions which can be made more precise for explicit models. We will, in particular, consider the GARCH$(p,q)$ model. We suppose,
\renewcommand{\labelenumi}{\textbf{A\arabic{enumi}}}
\begin{enumerate}
    \setcounter{enumi}{-1}
    \item $(\epsilon_t)_t$ is a causal, strictly stationary and ergodic
    solution of~\eqref{eq: intro, equation du
    modèle}.\label{hyp: MLE, 0}
\end{enumerate}

Let $\epsilon_1,\cdots,\epsilon_n$ denote observations of the process $(\epsilon_t)_t$. The true parameter of the model is denoted $\tau_0 =
\left(\theta_0',\psi_0'\right)'$, where $\theta_0$ is in
$\mathbb{R}^m$ and parameterizes the known function $g$,
$\psi_0=(\alpha_0,\beta_0,\mu_0)'$ is the parameter of the stable
density, the fixed parameter $\gamma_0=1$ being omitted. We still
denote by $f(.,\psi)=f(.,\alpha,\beta,\mu)$ the density of a stable law
and we also keep this notation for the stable characteristic
function. The parameter $\tau_0$ belongs to a parameter space
$\Gamma = \Theta\times A\times B\times C$ such that
$A\subset\left]0,2\right[$, $B\subset \left]-1,1\right[$ and
$C\subset\mathbb{R}$.

We define the criterion, for $\tau=(\theta',\psi')' \in\Gamma$:
\begin{equation*}
    \tilde{I}_n(\tau) = \frac1n \sum_{t=1}^n \tilde{l}_t(\tau)
        \quad \text{with} \quad \tilde{l}_t(\tau)
    =\frac{1}{2}\log\tilde{\sigma}_t^2(\theta) -
    \log f\left(\frac{\epsilon_t}{\tilde{\sigma}_t(\theta)},\psi\right),
\end{equation*}
where the $\tilde{\sigma}_t$ are recursively defined using some
initial values and 
\begin{equation*}
	\tilde{\sigma}_t^2(\theta) =
g\left(\epsilon_{t-1},\ldots,\epsilon_{1},\tilde{\epsilon}_0,
\tilde{\epsilon}_{-1},\ldots;\theta\right).
\end{equation*}
We also define
$\sigma_t^2(\theta) = g\left(\epsilon_{t-1},
\epsilon_{t-2},\ldots;\theta\right)$. Let $\tau_n$ be the MLE of
model~\eqref{eq: intro, equation du modèle} defined by:
\begin{equation}
    \label{eq: MLE, définition du MLE}
    \tau_n = \underset{\tau\in\Gamma}{\text{argmin}}\  \tilde{I}_n(\tau).
\end{equation}

\begin{sloppypar}
We define $\phi_{t,i}(\theta) =\frac{1}{\sigma_t^2(\theta)}
\frac{\partial \sigma_t^2}{\partial\theta_i}(\theta)$,
$\phi_{t,i,j}(\theta) =\frac{1}{\sigma_t^2(\theta)}\frac{\partial^2
\sigma_t^2}{\partial\theta_i\partial\theta_j}(\theta)$ and
$\phi_{t,i,j,k}(\theta) = \frac{1}{\sigma_t^2(\theta)}
\frac{\partial^3\sigma_t^2}{\partial\theta_i
\partial \theta_j \partial\theta_k}(\theta)$ and we state some
assumptions on the function $g$ and the parameter space $\Gamma$.
\end{sloppypar}

\renewcommand{\labelenumi}{\textbf{A\arabic{enumi}}}
\begin{enumerate}
    \setcounter{enumi}{0}
    \item There exists $\underline{\omega}>0$ such that, almost surely, for any
        $\theta\in\Theta$, $\sigma_t(\theta) >
        \underline{\omega}$. \label{hyp: MLE, 1 minoration}
    \item For $t>1$, $\underset{\theta\in\Theta}{\sup}\
        \left|\sigma_t^2(\theta) -
        \tilde{\sigma}_t^2(\theta)\right| < K \rho^t$, where $K$
        is a constant and $0<\rho<1$. \label{hyp: MLE, 2
        initial}
    \item $\forall t,\ \sigma_t(\theta) = \sigma_t(\theta_0)$
        implies $\theta=\theta_0$, a.s.\label{hyp: MLE, 3
        identifiabilité}
    \item The parameter space $\Gamma$ is a compact set and $\tau_0\in\Gamma$.
        \label{hyp: MLE, 4 compact}
    \item There exists $s>0$ such that $E |\epsilon_t|^s
        < +\infty$.\label{hyp: MLE, 5 moment}
%    \item $E\left\|\phi_t(\tau_0)\right\| < +\infty$, where
%        $\phi_t(\theta) = (\phi_{t,i})_{1\leq i\leq m}$.
%        \label{hyp: MLE 6}
%    \item $E\left\|\phi_t(\tau_0) \phi_t'(\tau_0)\right\| <
%        +\infty$.\label{hyp: MLE 7}
    \item For any compact subset $\Theta^*$ in the interior of
        $\Theta$ and for $(i,j,k)\in \left\{1,\ldots,m\right\}$,
        we have
        \begin{align*}
            &E\ \underset{\theta\in\Theta^*}{\sup}\ \left|\phi_{t,i}(\theta)\right| <+\infty,\
            E\ \underset{\theta\in\Theta^*}{\sup}\ \left|\phi_{t,i,j}(\theta)\right| < +\infty,\
            E\ \underset{\theta\in\Theta^*}{\sup}\
                \left|\phi_{t,i}(\theta) \phi_{t,j}(\theta)\right|<+\infty,\\
            &E\ \underset{\theta\in\Theta^*}{\sup}\ \left|\phi_{t,i}(\theta)
                \phi_{t,j,k}(\theta)\right| < +\infty,\
            E\ \underset{\theta\in\Theta^*}{\sup}\ \left|\phi_{t,i,j,k}(\theta)\right|
                < +\infty,\\
            &E\ \underset{\theta\in\Theta^*}{\sup}\ \left|
            \phi_{t,i}(\theta)\phi_{t,j}(\theta)\phi_{t,k}(\theta)\right|< +\infty.
        \end{align*}\label{hyp: MLE 8}
    \item For $t>1$, $\underset{\theta\in\Theta}{\sup}\left\|\frac{\partial \tilde{\sigma}_t^2}
        {\partial\theta}(\theta)- \frac{\partial{\sigma}_t^2}
        {\partial\theta} (\theta)\right\| < K\rho^t$, and
        $\underset{\theta\in\Theta}{\sup}\left\|\frac{\partial^2\tilde{\sigma}_t^2}
        {\partial\theta\partial\theta'}(\theta)-
        \frac{\partial^2{\sigma}_t^2}{\partial\theta\partial
        \theta'}(\theta)\right\| < K\rho^t$.\label{hyp: MLE 9}
    \item The components of $\frac{\partial\sigma_t^2}
        {\partial\theta}(\tau)$ are linearly independent.
        \label{hyp: MLE last}
\end{enumerate}

We prove that the estimator $\tau_n$ is CAN, the first result
establishes the consistency, then with additional assumptions, we
obtain the asymptotic normality of the estimator.

\begin{theo}
    \label{thm: MLE, consistency}
    Under Assumptions \textbf{A\ref{hyp: MLE, 0}-A\ref{hyp: MLE, 5 moment}},
    the estimator $\tau_n$ is consistent,
    \begin{equation}
        \label{eq: MLE, consistence}
        \tau_n \underset{n\rightarrow +\infty}{\longrightarrow} \tau_0 \quad \text{a.s.}
    \end{equation}
    If, in addition \textbf{A\ref{hyp: MLE 8}-A\ref{hyp: MLE last}}
    hold,
    \begin{equation}
        \label{eq: MLE, normalité asymptotique}
        \sqrt{n}(\tau_n-\tau_0) \arrowloi \mathcal{N}\left(0,J^{-1}\right),
    \end{equation}
    with $J = E\left[\frac{\partial^2 l_t(\tau_0)}
    {\partial\tau\partial\tau'}\right] = \E\left[\frac{\partial
    l_t}{\partial\tau}(\tau_0)\frac{\partial l_t}
    {\partial\tau'}(\tau_0)\right]$, where ${l}_t(\tau)
    =\frac{1}{2}\log{\sigma}_t^2(\theta) - \log
    f\left(\frac{\epsilon_t}{{\sigma}_t(\theta)},\psi\right)$.
\end{theo}

\begin{rem}
	The numerous assumptions of this theorem are due to the fact that Model~\eqref{eq: intro, equation du modèle} is very general. For more specific formulations, some of theses assumptions vanish. For exemple
    in the case of the GARCH$(p,q)$ model of Equation~\eqref{eq: intro,
    equation du modèle GARCH(p,q)}, Assumptions \textbf{A\ref{hyp:
    MLE, 1 minoration}, A\ref{hyp: MLE, 2 initial}, A\ref{hyp: MLE, 3
    identifiabilité}, A\ref{hyp: MLE, 5 moment},
    %A\ref{hyp: MLE 6},A\ref{hyp: MLE 7},
    A\ref{hyp: MLE 8}, A\ref{hyp: MLE 9}} and
    \textbf{A\ref{hyp: MLE last}} are obtained in \citet{FZ2004mle}.
\end{rem}

Concerning Assumption \textbf{A\ref{hyp: MLE, 0}}, in the case
of the GARCH$(p,q)$, we require the top Lyapunov exponent associated to
the model (see for instance \citet{berkes2003gps}) to be strictly
negative. In the case $p=q=1$, we draw the stationarity zones for
the parameters $a$ and $b$ for different values of $\alpha$. Here,
we use a symmetric stable distribution (i.e. $\beta=0$). In Figure
\ref{fig: MLE, stationarity}, we numerically obtained the strict
stationarity zones which for each $\alpha$, is the area under the
curve. If $\alpha=2$, this is the stationarity zone for a GARCH(1,1)
model with Gaussian innovation but in the case $\alpha<2$, the
strict stationarity zone becomes smaller as $\alpha$ decreases. This
can be explained by the fact that the smaller $\alpha$, the
thicker the tails. Then if the parameters $a$ and $b$ take too
large values, the persistence of $\sigma_t$ is too strong and
$\sigma_t$ explodes to infinity.

\begin{figure}[H]
    \centering
	\includegraphics[width=0.7\textwidth]{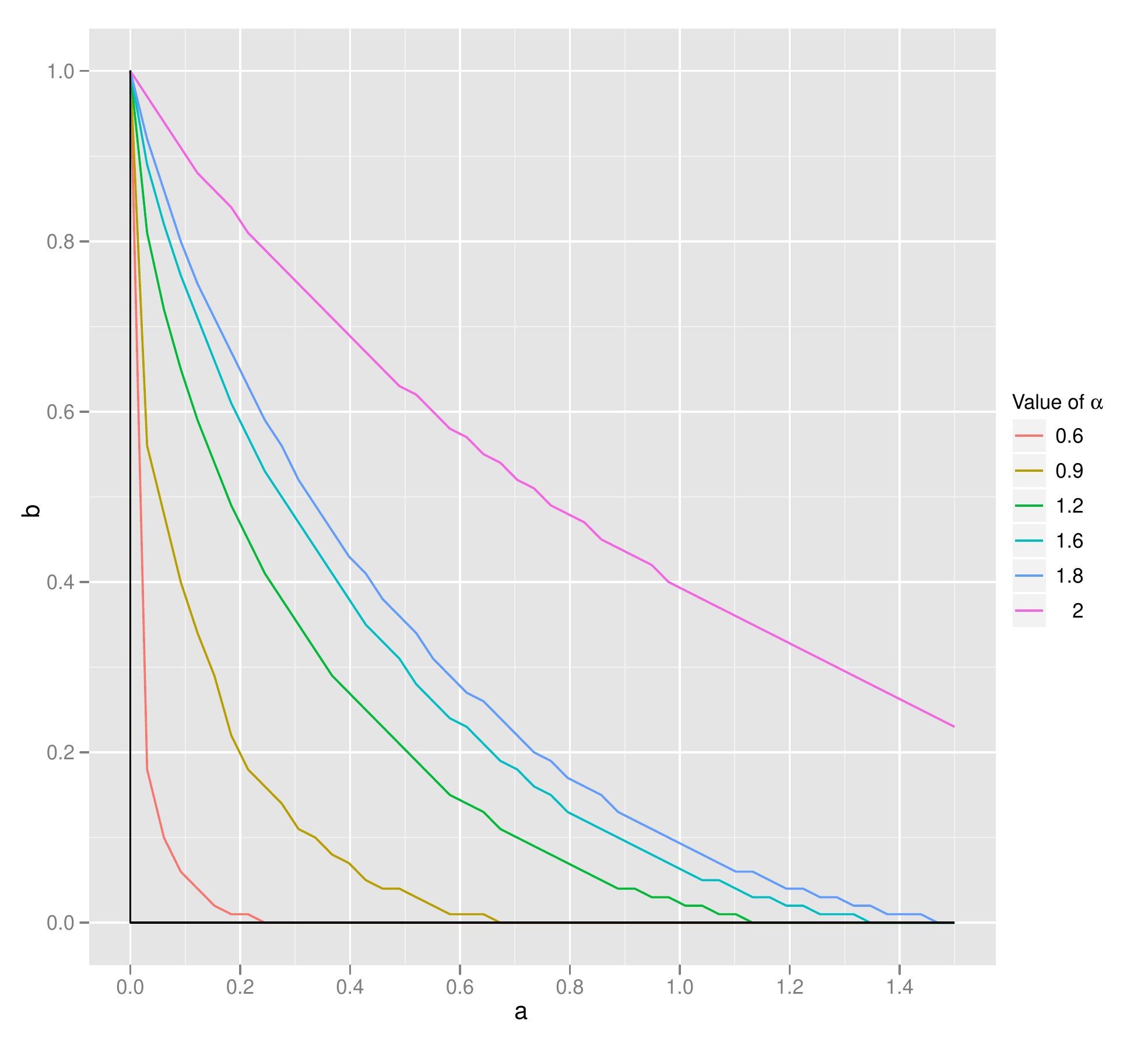}
    \caption{Strict stationarity zones for a GARCH(1,1) model with
    $\alpha$-stable innovation. The curves correspond, in decreasing order to $\alpha=2$,
    ..., $\alpha=0.6$.}
    \label{fig: MLE, stationarity}
\end{figure}

%\begin{figure}[H]
%    \centering
%    \includegraphics[width=0.7\textwidth]{../graphes/stationarity.pdf}
%    \caption{Strict stationarity zones for a GARCH(1,1) model with
%        $\alpha$-stable innovation}
%    \label{fig: MLE, stationarity}
%\end{figure}

%\emph{RAJOUTER DES REMARQUES, DES EXEMPLES, ...}

\section{When the innovation process converges in distribution to a stable distribution}
%[$(\eta_{nt})_t$ converges un distribution to a stable law]
\label{sec: pseudo GARCH}

In this section, for clarity purpose, we will enunciate the results
for a GARCH$(p,q)$ model (Model~\eqref{eq: intro, equation du modèle
GARCH(p,q)}). The same results could be obtained for a more general model but at the cost of some technical assumptions on the function $g$. We write a different version of Model~\eqref{eq: intro,
equation du modèle GARCH(p,q)} with an innovation process
$(\eta_{nt})$ which now depends on $n$. We have,
\begin{equation}
    \label{eq: pseudo GARCH, modèle}
    \left\{
        \begin{aligned}
            \epsilon_{nt} &= \sigma_{nt}\eta_{nt}\\
            \sigma_{nt}^2 &= \omega_0 + \sum_{i=1}^q a_{0i}\epsilon_{nt-i}^2 +
            \sum_{j=1}^p b_{0j}\sigma_{nt-j}^2,\quad \forall t\in\mathbb{Z},\
            \forall n\in\mathbb{N},
        \end{aligned}
    \right.
\end{equation}
where the process $(\eta_{nt})_t$ is iid with p.d.f. $f_n$ and
converges in distribution toward a stable variable with parameter
$\psi_0 = (\alpha_0,\beta_0,\mu_0)'$. This assumption will be made explicit below. As in Section~\ref{sec: stable
garch}, the parameter $\gamma_0$ is omitted and fixed to 1 for
identifiability reasons. The true parameter of the model is $\tau_0
= \left(\theta_0',\psi_0'\right)'$, where $\theta_0 =
\left(\omega_0,a_{01},\ldots,a_{0q},b_{01}, \ldots,b_{0p}\right)'$
belongs to a parameter space $\Theta\subset(0,+\infty) \times
\left[0,+\infty\right)^{p+q}$. The parameter $\tau_0$ belongs to
$\Gamma = \Theta\times A\times B\times C$, with $A, B, C$ as in
Section~\ref{sec: stable garch}. Let the polynomials
$\mathcal{A}_\theta(z) = \sum_{i=1}^q a_i z^i$ and
$\mathcal{B}_\theta(z) = 1-\sum_{j=1}^p b_j z^j$ where
$\theta=(\omega,a_1,\ldots,a_q,b_1,\ldots,b_p)'$. For $\theta$ such
that $\sum_{j=1}^p b_j <1$ and $b_j\geq 0$ for
$j\in\left\{1,\ldots,p\right\}$, define the function
$\sigma_{nt}^2(\theta) = \frac{\omega}{\mathcal{B}_\theta(1)}
+\mathcal{B}_\theta^{-1}(L)\mathcal{A}_\theta(L)\epsilon_{nt}^2$,
where $L$ denotes the lag operator.

We suppose that the process $(\eta_{nt})_t$ is iid for every
$n\in\mathbb{N}$, but we need a stronger assumption. Define
for $t\in\mathbb{Z}$, $\mathcal{F}_t=\sigma\left(
\cup_{n_\in\mathbb{N}} \mathcal{F}_{nt}\right)$, where
$\mathcal{F}_{nt} = \sigma\left\{\eta_{nu};u\leq t-1\right\}$ and
suppose that for any $t\in\mathbb{Z}$ and for any $n\in\mathbb{N}$,
$\eta_{nt}$ is independent of $\mathcal{F}_t$.

We define a pseudo maximum likelihood estimator. The density of
$(\eta_{nt})_t$ is not specified but we suppose the convergence of
this process to the stable distribution with p.d.f. $f(.,\psi_0)$,
where $\psi_0$ is an unknown parameter. We use this density to build
a pseudo MLE. We have for $\tau\in\Gamma$,
\begin{equation*}
    \tilde{I}_n(\tau) = \frac{1}{n} \sum_{t=1}^{n} \tilde{l}_{nt}(\tau)
        \quad \text{with} \quad \tilde{l}_{nt}(\tau)
    =\log\tilde{\sigma}_{nt}(\theta) -
    \log f\left(\frac{\epsilon_{nt}} {\tilde{\sigma}_{nt}(\theta)},\psi
    \right),
\end{equation*}
where the $(\tilde{\sigma}_{nt}(\theta))_t$ are recursively defined
using some initial values. We define
\begin{equation}
    \tau_n = \underset{\tau\in\Gamma}{\text{argmin}}\  \tilde{I}_n(\tau).
\end{equation}
We have kept the same notations as in the previous section because
all the involved quantities are defined in the same way and play the
same role. The objects of this section simply display an additional
$n$ subscript.

We define $\gamma_n$ as the top Lyapunov exponent associated to
Model~\eqref{eq: pseudo GARCH, modèle} and $\gamma$ as the top
Lyapunov exponent associated to the model
\begin{equation}
    \label{eq: pseudo GARCH, modèle sans n}
    \left\{
        \begin{aligned}
            \epsilon_{t} &= \sigma_{t}\eta_{t}\\
            \sigma_{t}^2 &= \omega_0 + \sum_{i=1}^q a_{0i}\epsilon_{t-i}^2 +
            \sum_{j=1}^p b_{0j}\sigma_{t-j}^2,\quad \forall t\in\mathbb{Z},
        \end{aligned}
    \right.
\end{equation}
with $\eta_t\sim S(\psi_0)$. The top Lyapunov exponent will be more
precisely defined in Section~\ref{sec: Proof}.

\renewcommand{\labelenumi}{\textbf{B\arabic{enumi}}}
\begin{enumerate}
    \item \label{hyp: pseudo GARCH Compact}
        $\tau_0\in\Gamma$ and $\Gamma$ is a compact.
    \item \label{hyp: pseudo GARCH, stationnarité}
        $\gamma<0$ and $\forall \theta\in\Theta,\ \sum_{j=1}^p
        b_j < 1$.
    \item \label{hyp: pseudo GARCH, convergence unif eta n}
        There exists $\delta>1$ such that for any $n\in\mathbb{N}$,
        $\E|\eta_{nt}|^\delta < +\infty$ and
        $\underset{x\in\mathbb{R}}{\sup}\ (1+|x|)^\delta
        |f_n(x)-f(x,\psi_0)| \rightarrow 0$.
%
%    $\underset{|x|\geq 1}{\sup}\ |x|^{\delta}
%    |f_n(x)-f(x,\psi_0)| \rightarrow 0$ and
%    $\underset{|x|<1}{\sup}\ |f_n(x)-f(x,\psi_0)| \rightarrow 0$
%    for any $0\leq s\leq
%    \delta$, we have $\underset{x}{\sup}\ |x|^{s}
%    |f_n(x)-f(x,\psi_0)| \rightarrow 0$, when $n\rightarrow
%    +\infty$.
    \item \label{hyp: pseudo GARCH, polynome A et B}
        If $p>0$, $\mathcal{A}_{\theta_0}(z)$ and
        $\mathcal{B}_{\theta_0}(z)$ have no common roots,
        $\mathcal{A}_{\theta_0}(1)\neq 0$ and $a_{0q}+b_{0p}\neq
        0$.
    \item \label{hyp: pseudo GARCH, mélange}
        We have $\underset{n\in\mathbb{N}}{\sup}\
        \alpha_{\epsilon_n}(h) \leq K\rho^h$, where
        $\alpha_{\epsilon_{n}}(h)$ for $h\in\mathbb{N}$ is the
        strong mixing coefficient of the process
        $(\epsilon_{nt})_t$.
        %and $Q(h)$ is such as $\sum_{h\geq 0} < +\infty$.
        %The processus $(\epsilon_{nt})$ is geometrically alpha-mixing uniformly with respect to n, that is that there exist $K>0$ and $\rho>0$ such that for any $h\in\mathbb{N}$, we have $\underset{n\in\mathbb{N}}{\sup}\ \Cov\left(\epsilon_{nt},\epsilon_{nt-h}\right) < K\rho^h$.
    \item\label{hyp: pseudo GARCH, interior}
        $\tau_0\in\stackrel{\circ}{\Gamma}$, where
        $\stackrel{\circ}{\Gamma}$ denotes the interior of
        $\Gamma$.
\end{enumerate}

\begin{theo}
    \label{theo: pseudo GARCH CAN}
    Under Assumptions \textbf{B\ref{hyp: pseudo GARCH
    Compact}-B\ref{hyp: pseudo GARCH, mélange}}, the estimator $\tau_n$
    is
    consistent,
    \begin{equation}
        \label{eq: pseudo GARCH consistence}
        \tau_n\underset{n\rightarrow+\infty}{\longrightarrow} \tau_0,\ a.s.
    \end{equation}
    If, in addition \textbf{B\ref{hyp: pseudo GARCH, interior}} holds,
    \begin{equation}
        \label{eq: pseudo GARCH AN}
        \sqrt{n}(\tau_n-\tau_0)\arrowloi \mathcal{N}\left(0,J^{-1}\right),
    \end{equation}
    with $J = E\left[\frac{\partial^2 l_t(\tau_0)}
    {\partial\tau\partial\tau'}\right] = \E\left[\frac{\partial
    l_t}{\partial\tau}(\tau_0)\frac{\partial l_t}
    {\partial\tau'}(\tau_0)\right]$, where ${l}_t(\tau)$ is defined in the proofs.
    %
%    $=\frac{1}{2}\log{\sigma}_t^2(\theta) - \log
%    f\left(\frac{\epsilon_t}{{\sigma}_t(\theta)},\psi\right)$.
\end{theo}

\begin{rem}
    The variance-covariance matrix in Theorem~\ref{theo: pseudo GARCH
    CAN} is the same as in Theorem~\ref{thm: MLE, consistency}. There is
    asymptotically no cost for not specifying the true distribution of
    the innovation and instead assuming that the process converges in distribution
    to a stable law.
\end{rem}

\begin{rem}
    The required assumptions for this result are very mild,
    \textbf{B\ref{hyp: pseudo GARCH Compact}, B\ref{hyp: pseudo GARCH,
    stationnarité}, B\ref{hyp: pseudo GARCH, polynome A et B}} and
    \textbf{B\ref{hyp: pseudo GARCH, interior}} are also needed for the
    classical Gaussian QML. Assumption \textbf{B\ref{hyp: pseudo GARCH, convergence unif eta n}} is specific to the problem and is verified in the case described hereafter where the innovation can be written as the sum of an iid process,
    %
    %    , it will
    %    be obtained with Theorem~\ref{thm: stable, GTCL} and Theorem
    %   ~\ref{thm: stable, GTCL for densities}. If the innovation process can
    %    be written, for any $n\in\mathbb{N}$ and any $t\in\mathbb{Z}$
    \begin{equation*}
        \eta_{nt} = \frac{1}{k_n^{1/\alpha}}\sum_{i=1}^{k_n}\nu_{it},
    \end{equation*}
    where for any $t\in\mathbb{Z}$, $(\nu_{it})_i$ is iid and where
    $(k_n)_n$ is an increasing sequence of integers and with
    $\alpha\in(1,2)$ such that there exist $\ K_1$ and $K_2$ such that
    \begin{align*}
        \PP\left[\nu_{1t}>x\right]&\sim K_1 x^{-\alpha},\
        \mbox{when}\ x\rightarrow +\infty,\\
        \PP\left[\nu_{1t}<x\right]&\sim K_2 x^{-\alpha},\
        \mbox{when}\ x\rightarrow -\infty,
    \end{align*}
    then if the density of $\nu_{it}$ satisfies the assumptions of
    Theorems~\ref{thm: stable, GTCL for densities} and~\ref{thm, stable
    basu}, the Assumption \textbf{B\ref{hyp: pseudo GARCH, convergence
    unif eta n}} is verified by Theorem~\ref{thm, stable basu}.
\end{rem}

\begin{rem}
    In the Gaussian QML case, the asymptotic inverse variance-covariance
    matrix $J$ depends on the unobserved distribution of the process
    $(\eta_t)_t$. Here the matrix $J$ depends on the limit in distribution
    of the innovation process $(\eta_{nt})_t$. We can define an estimator for the matrix $J$, based on the process $(\epsilon_{nt})_t$ and prove that this estimator is consistent.
\end{rem}

\begin{rem}
	About Assumption \textbf{B\ref{hyp: pseudo GARCH, mélange}}, for each value of $n$, the fact that there exist constants $K$ and $\rho$ such that $\alpha_{\epsilon_n}(h) \leq K\rho^h$ has been proved by \citet{boussama71998}. We only assume that this is also true for $\underset{n\in\mathbb{N}}{\sup}\ \alpha_{\epsilon_n}(h)$.
\end{rem}

\begin{theo}
    \label{thm: consistence de J_n} Define $J_n = \frac1n \sum_{t=1}^n
    \frac{\partial^2 \tilde{l}_{nt}}{\partial\tau\partial\tau'}(\tau_n)$. With
    the assumptions of Theorem~\ref{theo: pseudo GARCH CAN}, we have
    \begin{equation}
        \label{eq: consistence J_n}
        J_n \underset{n\rightarrow+\infty}{\longrightarrow} J,\quad \mbox{a.s.}
    \end{equation}
\end{theo}

\section{Numerical experiments}

In this section, we describe a simulation experiment which aims at studying the
behavior of the pseudo MLE for finite samples, and for an innovation
process whose distribution is close to a stable distribution. We use the algorithm of \citet{chambers1976method} to simulate stable processes and
Proposition~\ref{prop: développement en série} to compute the stable
density.

We want to verify that even if the model is
misspecified, that is if we use a stable MLE when the true
distribution of the innovation process is not stable, the GARCH coefficients are still correctly estimated. We use a Student distribution with degree of
freedom $\alpha$ (which by Theorem~\ref{thm: stable, GTCL} is in the
domain of attraction of a stable law of parameter $\alpha$) to build an innovation process of the form
\begin{equation*}
    \eta_t^{(K)} = \frac{1}{K^{1/\alpha}} \sum_{k=1}^K \nu_{k,t},\quad
    \mbox{with}\quad (\nu_{k,t})_k\stackrel{iid}{\sim} t_\alpha.
\end{equation*}
Using the results of Section~\ref{sec: stable}, we obtain that, when $K$ tends to infinity, $\eta_t^{(K)}$ converges in distribution toward an alpha stable law. The problem is that, for identifiability reason, the parameter $\gamma$ of the stable distribution cannot be estimated and has to be fixed to $1$. When $K=+\infty$, the process $\left(\eta_t^{(+\infty)}\right)_t$ is alpha stable with parameter $\psi =\left(\alpha,\beta,\mu,\gamma\right)'$. The parameter $\psi$ depends on the degree of freedom of the Student process $(\nu_{k,t})_k$ and can be calculated. For a generic case, we have $\gamma \neq 1$.
If we estimate a GARCH model with innovation process $\left(\eta_t^{(+\infty)}\right)_t$ using a stable pseudo MLE method, we would obtain estimates of the parameters of the same model but written under a different identifiability assumption. For example, if we aim to estimate the model
\begin{equation*}
	\left\{
		\begin{aligned}
			\epsilon_t &= \sigma_t \eta_t^{(+\infty)}\\
			\sigma_t^2 &= \omega_0 + a_0\epsilon_{t-1}^2 + b_0\sigma_{t-1}^2,
		\end{aligned}\right.
\end{equation*}
the stable pseudo MLE defined in the previous sections will converge toward $\theta_0^* = \left( \frac{\omega_0}{\gamma^2}, \frac{a_0}{\gamma^2}, b_0\right)'$, see \citet{francq2011two} for more details on reparametrization of GARCH models. Note that the estimation of the ``GARCH'' parameter $b_0$ is not affected by the identification problem. In order to compare estimates of the same quantity, it is thus important that the model is similarly identified for each value of $K$. Thereafter, we use the following identifiability condition. 
\begin{itemize}
	\item If the innovation process $(\eta_t)_t$ is stable, then it is stable with parameter $\gamma_0=1$ (we recall that, if $X$ is stable with parameter $\gamma = \gamma_0>0$ then $\frac{X}{\gamma_0}$ is stable with parameter $\gamma=1$).
	\item It the innovation process $(\eta_t)_t$ is not stable, we require that, among the family of stable distributions, the closest distribution to the distribution of $(\eta_t)_t$ in the sense of the Kulback-Leibler distance is stable with parameter $\gamma_0=1$.
\end{itemize}
Thus, for each $K$, we estimate the quantity $j_K$, defined such that the innovation process defined by $\eta_{t,K} =  \frac{1}{j_K K^{1/\alpha}} \sum_{k=1}^K \nu_{k,t}$ satisfies the identifiability assumption. It is important to note that if we use  another normalizing constant than $j_K$, the results of the estimation by stable pseudo ML are as efficient as in the case where we use $j_K$, the model is simply written under a different identifiability condition.

We generated 1000 samples of size $n=1000$ for different values of $K$ ($K=500,\ K=1000,\ K=10000,\ K=100000$ and $K=\infty$) of the following model and
estimate its parameters by stable MLE (or pseudo MLE).
\begin{equation}
    \left\{
    \begin{aligned}
        \epsilon_t &= \sigma_t\eta_{t,K}\\
        \sigma_t^2 &= 0.01+0.02\epsilon_{t-1}^2+0.7\sigma_{t-1}^2\\
        \eta_{t,K} &= \frac{1}{j_K K^{1/\alpha}} \sum_{k=1}^K \nu_{k,t}.
    \end{aligned}\right.
    \label{eq: modele simu}
\end{equation}
We can summarize the simulation scheme with the following steps. For a parameter $\theta_0$, for $K>0$ and for a student distribution of degree $\alpha$,
\begin{itemize}
	\item Step 1: we simulate $1000$ samples of the variable $\left(\frac{1}{ K^{1/\alpha}} \sum_{k=1}^K \nu_{k,t}\right)_t$, then we fit a stable distribution on each sample. For each sample $s$, we denote by $\psi_s = \left(\alpha_s,\beta_s,\mu_s,\gamma_s\right)'$ the results of this estimation.
	\item Step 2: we compute $j_K = \frac1{1000} \sum_{s=1}^{1000} \gamma_s$.
	\item Step 3: we draw $1000$ samples of Model~\eqref{eq: modele simu}.
	\item Step 4: for each sample $s$, we estimate $\tau_n^{(s)} = \left(\theta_n^{(s)'},\psi_n{(s)'}\right)'$ by stable PMLE.
\end{itemize}

\vspace{0.5cm}

%In the stable pseudo MLE, the scale parameter of the distribution of
%the innovation process is fixed to 1, it is not estimated. In order to obtain estimators which can be compared with each others, we have correct the identification bias of parameters $\omega$ and $a$. The value of this bias can be found empirically by estimating the parameters of a stable distribution on the process $\left(\frac{1}{n^{1/\alpha}} \sum_{k=1}^K \nu_{k,t}\right)_t$. 
The results of these estimations are presented in Table~\ref{tab: result simul}. For each of the six parameters (three for
the GARCH dynamic and three for the stable distribution), we give
the quotient of the Root Mean Squared Errors (RMSE) of the corrected
stable pseudo MLE and of the RMSE of the MLE, corresponding to the case
$K=\infty$. This statistic is given by $Q_i^K =
\frac{RMSE_i^{K=+\infty}}{RMSE_i^{K=K}}$ where the RMSE for $K$ can
be obtained by, for $i\in\left\{1,\cdots,6\right\},\ RMSE_i^K =
\frac1{1000} \sum_{s=1}^{1000} \left(\theta_{n,i}^{(s)}-\theta_{0,i}\right)^2$. The
greater $Q_i$ is, the better the stable pseudo MLE is with respect to the
MLE. In this simulation framework, we do not compare different methods of estimation. We use the same method but applied to different data generating processes.
Table \ref{tab: result simul} shows that, concerning parameters $\omega$, $a$, $b$ and $\alpha$, the $(Q_i)_{i\in\{1,\cdots,4\}}$ increase with $K$. For large values of $K$, the RMSE of the misspecified model is quite close to the RMSE of the asymptotic case, i.e. the well specified case. Table \ref{tab: result simul} also indicates that the result of the simulation for parameters $\beta$ and $\delta$ are more surprising. Indeed, the RMSE of the estimation of the asymmetry parameter $\beta$ is smaller when $K$ is small. But a good estimation of the parameter $\beta$ is not of much information if the parameter $\alpha$ is not well estimated. Finally, for large values of $K$, we find that our estimator is not much affected by the specification error on the density used to compute the likelihood.

\begin{table}[H]
\begin{center}
%\resizebox{\textwidth}{!}{
    \begin{tabular}{|c|cccccc|cccccc|} \hline
    &\multicolumn{6}{c|}{$\alpha=1.6$}&\multicolumn{6}{c|}{$\alpha=1.4$}\\
    $K$&10&50&500&1000&10000&100000&10&50&500&1000&10000&100000\\
  \hline
$w$ & 0.34 & 0.59 & 0.79 & 0.66 & 0.89 & 0.86 & 0.56 & 0.68 & 0.80 & 0.83 & 0.78 & 0.91 \\ 
  $a$ & 0.12 & 0.24 & 0.50 & 0.54 & 0.82 & 0.83 & 0.35 & 0.61 & 0.85 & 0.74 & 0.81 & 0.97 \\ 
  $b$ & 0.42 & 0.61 & 0.76 & 0.76 & 0.94 & 0.99 & 0.76 & 0.92 & 0.97 & 0.93 & 0.94 & 0.97 \\ 
  $\alpha$ & 0.17 & 0.31 & 0.58 & 0.70 & 0.91 & 1.03 & 0.57 & 0.84 & 0.93 & 0.94 & 0.95 & 0.96 \\ 
  $\beta$ & 2.45 & 1.45 & 1.27 & 1.22 & 1.05 & 1.01 & 1.29 & 0.92 & 0.84 & 1.12 & 0.86 & 1.04 \\ 
  $\delta$ & 1.26 & 0.96 & 0.96 & 0.87 & 0.89 & 0.86 & 1.07 & 0.96 & 0.90 & 1.09 & 0.90 & 0.93 \\ 
   \hline
    \end{tabular}%}
\vspace{3mm} \caption{Ratio of RMSE for the parameters of the model for
several values of $K$ and $\alpha$.} \label{tab: result simul}
\end{center}
\end{table}

\section{Application to financial data}

In this section, we consider the daily returns of several indices
and currency rates, namely the EURUSD, JPYUSD, DJA, DJI, DJT, DJU,
CAC, FTSE, NIKKEI, DAX, S\&P50 and the SMI. A GARCH(1,1) model with
stable innovations is estimated on each of these series. The samples
extend from January 1, 2008 to December 31, 2010. The estimated
$\alpha$'s are lower for the period before 2008 so we only kept
three years of data. Table~\ref{tab: real data result} shows the
results of these estimations with the standard deviation in
parenthesis. We can see that even if the GARCH modeling explain a
fraction of the leptokurticity of the series, the residuals still
possess heavy tails since in most cases $\alpha$ is around 1.8 and
thus different from 2 (except for the NIKKEI). When $\alpha=2$, the
parameter $\beta$ cannot be identified.

\begin{table}[H]
    \begin{center}
    \resizebox{\textwidth}{!}{
    \begin{tabular}{|l|cc|cc|cc|cc|cc|cc|} \hline
        \multicolumn{1}{|c|}{Index}&\multicolumn{2}{|c|}{$\omega\times 10^5$}&\multicolumn{2}{c|}{$a\times 10^2$}&\multicolumn{2}{c|}{$b$}&\multicolumn{2}{c|}{$\alpha$}&\multicolumn{2}{c|}{$\beta$}&\multicolumn{2}{c|}{$\mu$}\\
        \hline
        EURUSD&~0.133&(~0.047)&~4.120&(~0.720)&~0.877&(~0.025)&~1.900&(~0.020)&-0.007&(~0.480)&-0.007&(~0.075)\\
        JPYUSD&~0.578&(~0.270)&~3.370&(~0.770)&~0.667&(~0.130)&~1.780&(~0.037)&-0.148&(~0.190)&-0.137&(~0.062)\\
        INRUSD&~0.042&(~0.016)&~2.300&(~0.630)&~0.912&(~0.024)&~1.820&(~0.077)&~0.242&(~0.230)&-0.016&(~0.067)\\
        DJA&~0.050&(~0.044)&~4.610&(~0.540)&~0.895&(~0.014)&~1.820&(~0.110)&-0.039&(~0.280)&~0.028&(~0.071)\\
        DJI&~0.046&(~0.038)&~4.670&(~0.400)&~0.893&(~0.011)&~1.840&(~0.083)&-0.501&(~0.270)&~0.114&(~0.071)\\
        DJT&~0.071&(~0.070)&~3.820&(~0.560)&~0.916&(~0.014)&~1.880&(~0.087)&-0.060&(~0.330)&~0.060&(~0.067)\\
        DJU&~0.095&(~0.045)&~4.690&(~0.700)&~0.887&(~0.018)&~1.910&(~0.068)&-0.296&(~0.360)&~0.040&(~0.064)\\
        CAC40&~0.221&(~0.092)&~2.970&(~0.550)&~0.914&(~0.016)&~1.880&(~0.072)&-0.220&(~0.210)&-0.010&(~0.058)\\
        FTSE&~0.156&(~0.067)&~3.660&(~0.610)&~0.898&(~0.019)&~1.820&(~0.050)&-0.222&(~0.230)&~0.059&(~0.060)\\
        NIKKEI&~0.376&(~0.140)&~5.700&(~0.490)&~0.863&(~0.021)&~2.000&(~0.090)&~NA&(~NA)&-0.021&(~0.053)\\
        DAX&~0.136&(~0.051)&~2.830&(~0.510)&~0.922&(~0.014)&~1.860&(~0.079)&-0.252&(~0.270)&~0.034&(~0.065)\\
        S\&P500&~0.909&(~0.250)&~3.870&(~0.780)&~0.766&(~0.055)&~1.770&(~0.057)&-0.162&(~0.180)&~0.103&(~0.062)\\
        SMI&~0.112&(~0.049)&~4.850&(~0.470)&~0.876&(~0.015)&~1.850&(~0.077)&-0.255&(~0.250)&~0.028&(~0.064)\\
        \hline
    \end{tabular}
    }
    \end{center}
    \caption{Estimation of a GARCH(1,1), standard errors in parenthesis.}
    \label{tab: real data result}
\end{table}

These estimations can be used to compute Value-at-Risk (VaR). If
$\epsilon_t$ is the return of the series, the VaR with coverage
probability $p$ at time $t$ is defined as the quantity $\VaR_t(p)$
satisfying
\begin{equation*}
    \PP_t[\epsilon_t \leq \VaR_t(p)] = p,
\end{equation*}
where $\PP_t$ is the probability measure conditionally to the time
$t-1$ information set. Using a conditionally heteroscedastic model
with stable innovation, if $\tau_n=\left(\theta_{n}',
\psi_{n}'\right)'$ is the MLE (or pseudo MLE if the innovation is
not stable but assumed to be close to a stable distribution), we
have
\begin{equation*}
    \VaR_t(p) = \tilde{\sigma}_{t-1}(\theta_n) F^{\leftarrow}(p,\psi_n),
\end{equation*}
where $F^{\leftarrow}(.,\psi)$ is the quantile function of a stable
distribution of parameter $\psi$ (with $\gamma=1$). We compare this
\emph{stable} VaR to the VaR computed using a GARCH(1,1) model,
estimated with the Gaussian QMLE. We compute a Gaussian QMLE on the
indices used in Table~\ref{tab: real data result} and obtain the
Gaussian counterparts of the parameters in this table.

Then we compute the Gaussian VaR and the stable VaR on an outsample
data set (from January 1, 2011 to January 31, 2012). We give the
results for the VaR of level $1\%$ and $5\%$ in Table~\ref{tab:
result VaR}.

\begin{table}[H]
\begin{center}
%    \resizebox{\textwidth}{!}{
    \begin{tabular}{|l|cc|cc|} \hline
        \multicolumn{1}{|c|}{Index}&\multicolumn{2}{|c|}{Level=0.01}&\multicolumn{2}{|c|}{Level=0.05}\\
        \hline
        &Stable&Gaussian&Stable&Gaussian\\
        \hline
        EURUSD&0.0108&0.0180&0.0755&0.0791\\
        JPYUSD&0.0035&0.0035&0.0177&0.0213\\
        INRUSD&0.0106&0.0142&0.0426&0.0355\\
        DJA&0.0221&0.0221&0.0551&0.0515\\
        DJI&0.0147&0.0221&0.0551&0.0588\\
        DJT&0.0294&0.0331&0.0588&0.0478\\
        DJU&0.0110&0.0110&0.0551&0.0588\\
        CAC40&0.0108&0.0179&0.0645&0.0717\\
        FTSE&0.0037&0.0110&0.0625&0.0662\\
        NIKKEI&0.0114&0.0114&0.0342&0.0342\\
        DAX&0.0072&0.0143&0.0717&0.0789\\
        S\&P500&0.0074&0.0221&0.0699&0.0588\\
        SMI&0.0109&0.0181&0.0688&0.0688\\
        \hline
        \textbf{Mean}&\textbf{0.0118}&\textbf{0.0168}&\textbf{0.0563}&\textbf{0.0563}\\
        \hline
    \end{tabular}%}
\vspace{3mm} \caption{Frequency of hits for the Gaussian VaR and the
stable VaR} \label{tab: result VaR}
\end{center}
\end{table}

The two methods give very close results for $p=0.05$, but the
Gaussian method seems unable to explain the extremes of the
distribution and the stable VaR seems to give better results for
$p=0.01$. In this case and for almost every index, the Gaussian VaR is underestimated. There are too many hits in the sample. We can conclude that the residuals of the GARCH model estimated by Gaussian QMLE are too leptokurtic to be explained by a Gaussian distribution. To conclude, the stable distribution seems to do a better job to explain the tails of the studied financial series.

%\section{Conclusion}
%
%This article derives

\newpage
\section{Proofs}
\label{sec: Proof}

Throughout the proofs and the paper, we denote by $K$ and $\rho$
generic constants whose values $K>0$ and $0<\rho<1$ can vary from
line to line.

\subsection{Proof of the consistency in Theorem~\ref{thm: MLE, consistency}}

Let $I_n$ (resp. $l_t$) be the equivalent of $\tilde{I}_n$ (resp. $\tilde{l}_t$)
when an infinite past is known,
\begin{equation*}
    {I}_n(\tau) = \frac1n \sum_{t=1}^n {l}_t(\tau)
        \quad \text{with} \quad {l}_t(\tau) =\frac{1}{2}\log{\sigma}_t^2(\theta)
        - \log f\left(\frac{\epsilon_t}{{\sigma}_t(\theta)},\psi\right).
\end{equation*}
We first prove that the initial values are asymptotically
negligible, that is
\begin{equation}
    \label{eq:négligeabilité des valeurs initiales pour la consistence}
    \underset{n\rightarrow +\infty}{\lim}\
    \underset{\tau\in\Gamma}{\sup}\
    \left|I_n(\tau)-\tilde{I}_n(\tau)\right| = 0 \quad \text{a.s.}
\end{equation}
We have,
\begin{equation*}
    \underset{\tau\in\Gamma}{\sup}\ \left|
    I_n(\tau)-\tilde{I}_n(\tau)\right| \leq \frac1n \sum_{t=1}^n
    \underset{\tau\in\Gamma}{\sup}\ \left|
    \frac{1}{2} \log\frac{\sigma_t^2(\theta)}{\tilde{\sigma}_t^2(\theta)} +
    \log\frac{f(\tilde{\eta}_t(\theta),\psi)}{f(\eta_t(\theta),\psi)}\right|,
\end{equation*}
where $\tilde{\eta}_t(\theta) =
\frac{\epsilon_t}{\tilde{\sigma}_t(\theta)}$ and ${\eta}_t(\theta) =
\frac{\epsilon_t}{{\sigma}_t(\theta)}$.

The function $f$ is infinitely differentiable with respect to $x$,
therefore, for $\tau=(\theta',\psi')'\in\Gamma$, we have
\begin{equation*}
    \left|\log f(\tilde{\eta}_t(\theta),\psi) - \log f(\eta_t(\theta),\psi)\right|<
    \underset{x\in\mathbb{R}}{\sup}\ \left|\frac{f'(x,\psi)}{f(x,\psi)}\right|
        \left|\tilde{\eta}_t(\theta) - \eta_t(\theta)\right|.
\end{equation*}
Next, using the asymptotic expansion~\eqref{eq:tail behavior
1}-\eqref{eq:tail behavior 2}, we obtain that $\frac{f'(x,\psi)}
{f(x,\psi)}$ tends to 0, when $|x|$ tends to infinity and thus that
$x\mapsto \frac{f'(x,\psi)}{f(x,\psi)}$ is bounded on $\mathbb{R}$. This can be obtained since for any $\psi\in A\times B\times C$, the support of the function $x\mapsto f(x,\psi)$ is $\mathbb{R}$. Under Assumption \textbf{A\ref{hyp: MLE, 4 compact}}, we obtain
$\underset{\tau\in\Gamma}{\sup}\ \underset{x\in\mathbb{R}}{\sup}\
\left|\frac{f'(x,\psi)} {f(x,\psi)}\right| < \infty$. Using Assumption \textbf{A\ref{hyp: MLE, 1 minoration}} and \textbf{A\ref{hyp: MLE, 2 initial}}, it can be seen that $\underset{\tau\in\Gamma}{\sup}\ \left| {\eta}_t(\theta) - \eta_t(\theta)\right| < K|\epsilon_t|$ and $ \underset{\tau\in\Gamma}{\sup}\
\left|\tilde{\sigma}_t(\theta)-\sigma_t(\theta)\right| <
K|\epsilon_t|\rho^t$. Thus we have $\underset{\tau\in\Gamma}{\sup}\
|\log f(\tilde{\eta}_t(\theta),\tau) - \log f(\eta_t(\theta),\tau)|<
K|\epsilon_t|\rho^t$ and finally
\begin{equation*}
    \underset{\tau\in\Gamma}{\sup}\
    \left|I_n(\tau) - \tilde{I}_n(\tau)\right| < \frac1n
    \sum_{t=1}^n K|\epsilon_t|\rho^t.
\end{equation*}
In view of the Markov inequality, the Borel-Cantelli Lemma, the existence
of a moment of order $s$ for the processus $(\epsilon_t)_t$
(Assumption \textbf{A\ref{hyp: MLE, 5 moment}}) and Assumption
\textbf{A\ref{hyp: MLE, 0}}, we obtain that $|\epsilon_t|\rho^t$
converges to 0 almost surely when $t$ tends to infinity. Then, using
the Cesàro Lemma, we obtain~\eqref{eq:négligeabilité des valeurs
initiales pour la consistence}.

We now prove that $E l_t(\tau) > E l_t(\tau_0)$, for $\tau\neq
\tau_0$.
\begin{align*}
    E l_t(\tau) - E l_t(\tau_0) &= E\left[\log\frac{\sigma_t(\theta_0)}{\sigma_t(\theta)}
        \frac{f(\eta_t(\theta),\psi)}{f(\eta_t,\psi_0)}\right]\\
    &\leq E\left[\frac{\sigma_t(\theta_0)}{\sigma_t(\theta)}
        \frac{f(\eta_t(\theta),\psi)}{f(\eta_t,\psi_0)}\right] -1\\
    &=E\left[E\left[\frac{\sigma_t(\theta_0)}{\sigma_t(\theta)}
        \frac{f(\frac{\sigma_t(\theta_0)}{\sigma_t(\theta)}\eta_t,\psi)}
        {f(\eta_t,\psi_0)}|\mathcal{G}_{t}\right]\right]-1\\
        &=0.
\end{align*}
To obtain the last equality, we used the fact that
$\frac{\sigma_t(\theta_0)}{\sigma_t(\theta)}$ is in $\mathcal{G}_{t}
= \sigma\left\{\eta_u;\ u\leq t-1\right\}$.

Now, we show that, if $E l_t(\tau) = E l_t(\tau_0)$, then
$\tau=\tau_0$ a.s. We have
$\frac{\sigma_t(\theta_0)}{\sigma_t(\theta)}
\frac{f(\frac{\sigma_t(\theta_0)}{\sigma_t(\theta)}\eta_t,\psi)}
{f(\eta_t,\psi_0)}=1$ a.s. Let $a_{t-1} =
\frac{\sigma_t(\theta_0)}{\sigma_t(\theta)}$. The variable $(\eta_t)_t$ has a Lebesgue density, this yields
\begin{equation}
    \label{eq:égalité en distribution}
    \forall x\in \mathbb{R},\ a_{t-1} f(a_{t-1}x,\psi) = f(x,\psi_0) \quad a.s.
\end{equation}
We define $X$ as a stable variable with parameters $\psi_0$, $X\sim
S(\psi_0)$, and $Y = a_{t-1} X$. Then, using~\eqref{eq:égalité en
distribution} we show that the pdf of $Y$ conditionally to
$\mathcal{G}_{t}$ is $f(x,\psi)$, thus $Y\sim S(\psi)$. Now, for
$u\in \mathbb{R}$, we write the characteristic function of $Y$ and
obtain
\begin{align*}
    E\left[e^{iuY}|\mathcal{G}_{t}\right] &= E\left[e^{iua_{t-1}X}|\mathcal{G}_{t}\right],
\end{align*}
that is $\varphi_{\psi}(u) = \varphi_{\psi_0}(a_{t-1} u)$. Applying
the modula to the previous equation yields
\begin{equation*}
    \forall u\in\mathbb{R},\ \exp\left\{-|u|^\alpha\right\} =
    \exp\left\{-|a_{t-1}u|^{\alpha_0}\right\}.
\end{equation*}
Therefore we have $\alpha=\alpha_0$ and we easily obtain
$\beta=\beta_0$ and $\mu=\mu_0$. We also obtain $a_{t-1}=1$ almost
surely and we deduce with Assumption \textbf{A\ref{hyp: MLE, 3
identifiabilité}} that $\theta=\theta_0$ a.s.

Now, for $\tau \in \Gamma$, let $V_k(\tau)$ be the open ball with
center $\tau$ and radius $1/k$, using~\eqref{eq:négligeabilité des
valeurs initiales pour la consistence}, it follows that
\begin{equation}
    \label{eq: consistence MLE, CI}
    \text{liminf}\ \underset{\tau^*\in V_k(\tau)\cap\Gamma}{\inf} \tilde{I}_n(\tau^*)
    = \text{liminf}\ \underset{\tau^*\in V_k(\tau)\cap\Gamma}{\inf} {I}_n(\tau^*),
\end{equation}
The ergodic theorem yields $\text{liminf}\
\underset{\tau^*\in V_k(\tau)\cap\Gamma}{\inf} {I}_n(\tau^*) \geq
E\left[\underset{\tau^*\in
V_k(\tau)\cap\Gamma}{\inf}l_t(\tau^*)\right]$. When $k\rightarrow+\infty$, then $E\left[\underset{\tau^*\in V_k(\tau)\cap\Gamma}{\inf}l_t(\tau^*)\right]$ tends toward $E\left[l_t(\tau)\right]$. Therefore, we have
\begin{equation*}
    \forall \tau \neq \tau_0,\ \exists V(\tau),\
    \text{liminf}\ \underset{\tau^*\in V_k(\tau)\cap\Gamma}{\inf} \tilde{I}_n(\tau^*)
    > El_t(\tau_0).
\end{equation*}
We conclude by a standard compactness argument, using Assumption
\textbf{A\ref{hyp: MLE, 4 compact}} and obtain~\eqref{eq: MLE,
consistence}.

\subsection{Proof of the asymptotic normality in Theorem~\ref{thm: MLE,
consistency}} \label{sec: proof of the asymptotic normality MLE}

\begin{lem}
    \label{lem: normalité asymptotique 1}
    Under the assumptions of Theorem~\ref{thm: MLE, consistency}, we
    have
    \begin{equation}
        \label{eq:convergence en loi  TCL}
        \sqrt{n} \frac{\partial I_n}{\partial \tau}(\tau_0)
        \arrowloi \mathcal{N}(0,J).
    \end{equation}
\end{lem}

\begin{proof}
    For $\lambda \in \mathbb{R}^{m+3}$, $n>0$ and $t>0$,
    let $\nu_{nt} = \frac{1}{\sqrt{n}}\lambda' \frac{\partial
    l_t}{\partial\tau}(\tau_0)$. We prove that $(\nu_{nt},\mathcal{G}_{t-1})$ is a
    martingale difference. We have, for $\tau=(\theta',\psi')'\in\Gamma$
    \begin{align}
        \frac{\partial l_t}{\partial\theta}(\tau) &= \frac{1}{2}\phi_t(\theta)Z_t(\tau),
            \label{eq: proof, dérivées prem 1}\\
        \frac{\partial l_t}{\partial \psi}(\tau) &=
            -\frac{\partial\log f}{\partial \psi}(\eta_t(\theta),\psi)
            \label{eq: proof, dérivées prem 2},
    \end{align}
    with $Z_t(\tau) =
    1+\eta_t(\theta)\frac{f'(\eta_t(\theta),\psi)}{f(\eta_t(\theta),\psi)}.$

    Since $\sigma_t^2(\theta_0) \in \mathcal{G}_{t}$, $\sigma_t(\theta_0)$
    and $\eta_t$ are independent and
    \begin{equation*}
        E\left|\frac{\partial l_t}{\partial\theta}(\tau_0)\right|=
        E\left|\phi_t(\theta_0)\right| E\left|Z_t(_au_0)\right|.
    \end{equation*}
    The function $x\mapsto 1+x\frac{f'(x,\psi_0)}{f(x,\psi_0)}$ is
    bounded, therefore we have $\E\left|Z_t(\tau_0)\right|
    <+\infty$. Moreover, with Assumption \textbf{A\ref{hyp: MLE 8}},
    we have $E\left|\phi_t(\theta_0)\right|<+\infty$. With~\eqref{eq:tail behavior
    3}-\eqref{eq:tail behavior 5}, we obtain
    $E\left|\nu_{nt}\right|<+\infty$.

    Now, we have
    \begin{equation*}
        E\left[\frac{\partial l_t}{\partial\theta}(\tau_0)|\mathcal{G}_t\right]
        =\phi_t(\theta_0)E\left[Z_t(\tau_0)\right].
    \end{equation*}
    And,
    \begin{equation*}
        E\left[Z_t(\tau_0)\right] = 1+\int_{\mathbb{R}} xf'(x,\psi_0)dx = 0.
    \end{equation*}
    We infer $E\left[\frac{\partial l_t} {\partial\theta}(\tau_0) \right]
    = 0$. Now, we prove $E\left[\frac{\partial l_t}
    {\partial\alpha}(\tau_0)|\mathcal{G}_t\right] = 0$. Note that
    \begin{equation*}
        E\left[\frac{\partial l_t}{\partial\alpha}(\tau_0)|\mathcal{G}_t\right]
        = -\int_{\mathbb{R}} \frac{\partial f}{\partial\alpha}(x,\psi_0) dx.
    \end{equation*}
    The function $x\mapsto f(x,\psi_0)$ is integrable, therefore
    \begin{equation*}
        \frac{\partial\varphi_{\psi_0}(u)}{\partial\alpha}
        =-\int_{\mathbb{R}} e^{iux} \frac{\partial f}{\partial\alpha}(x,\psi_0) dx,
        \quad \text{and} \quad
        \frac{\partial\varphi_{\psi_0}(0)}{\partial\alpha}
        = E\left[\frac{\partial l_t}{\partial\alpha}(\tau_0)\right].
    \end{equation*}
    We have, $\forall \alpha,\ \varphi_{\alpha,\beta_0,\mu_0,\gamma_0}(0) = 0$,
    thus,
    $\frac{\partial\varphi_{\psi_0}}{\partial\alpha}(0) = 0$.
    Using the same method for $\frac{\partial l_t}{\partial\beta}$,
    $\frac{\partial l_t}{\partial\mu}$ and $\frac{\partial l_t}{\partial\gamma}$
    we obtain $E\left[\nu_{nt}|\mathcal{G}_t\right] = 0$.

    We now prove that the covariance matrix of the vector of derivatives of $l_t$ is finite, we have
    \begin{equation*}
        E\left[\frac{\partial l_t}{\partial\theta}(\tau_0)
        \frac{\partial l_t}{\partial\theta'}(\tau_0)\right]
        =\frac{1}{4}E\left[Z_t^2(\tau_0)\right]E\left[\phi_t(\theta_0)\phi_t'(\theta_0)\right],
    \end{equation*}
    with $E\left[Z_t^2(\tau_0)\right] =
    1+2\int_{\mathbb{R}}xf'(x,\psi_0)dx +
    \int_{\mathbb{R}}
    x^2\frac{f^{'2}(x,\psi_0)}{f(x,\psi_0)} dx$.
    From the asymptotic expansions in~\eqref{eq:tail behavior 1}-\eqref{eq:tail behavior 5},
    we obtain
    $E\left[Z_t^2(\tau_0)\right]<+\infty$, then by
    Assumption \textbf{A\ref{hyp: MLE 8}}, it can be seen that $V\left[\frac{\partial l_t}{\partial\theta}(\tau_0)\right]$ is finite.

    Using the asymptotic expansion again, we have when $x\rightarrow\pm\infty$,
    for $\psi \in A\times B\times C$
    \begin{equation}
        \frac{\partial\log f}{\partial\alpha}(x,\psi) \sim K\log|x|,\
        \frac{\partial\log f}{\partial\beta}(x,\psi) \sim K,\
        \frac{\partial\log f}{\partial\mu}(x,\psi) \sim K x^{-1},
        \label{eq: proof, asymptotique du log}
    \end{equation}
    therefore, for $(i,j)\in\left\{1,\cdots,3\right\}$, we have
    $\E\left[\frac{\partial l_t}{\partial\psi_i} \frac{\partial
    l_t}{\partial\psi_j}\right]$ is finite.
    The very same reasoning applies for $E\left[\frac{\partial l_t}{\partial\psi}
    \frac{\partial l_t}{\partial\theta}\right]$ and we deduce $V\left[\frac{\partial l_t}
    {\partial\tau}\right] < +\infty$.

    We now show that this matrix is positive-definite.
    Suppose that $(u',v')' \in \mathbb{R}^{m+3}$ are such that $(u',v'). \frac{\partial l_t}{\partial\tau}(\tau_0) = 0$. We have
    \begin{equation}
        \label{eq:positivité de la matrice de variance 1}
        \left(1+\eta_t \frac{f'(\eta_t,\psi)}{f(\eta_t,\psi)}\right)
        \frac{1}{2}u'\phi_t(\theta_0)  =
        v'\frac{\partial \log f}{\partial\psi}(\eta_t,\psi).
    \end{equation}

    Now, it can be seen that $\left(1+\eta_t
    \frac{f'(\eta_t,\psi)}{f(\eta_t,\psi)}\right)\in\mathcal{G}_{t+1}$,
    $v'\frac{\partial \log f}{\partial\psi}
    (\eta_t,\psi)\in\mathcal{G}_{t+1}$ and
    $\frac{1}{2}u'\phi_t(\theta_0) \in \mathcal{G}_t$. Thus, we
    have $\frac{1}{2}u'\phi_t(\theta_0)=Q$ and
    \begin{equation*}
        Q\left(1+\eta_t \frac{f'(\eta_t,\psi)}{f(\eta_t,\psi)}\right)=
        v_1\frac{\partial \log f}{\partial \alpha}(\eta_t,\alpha,\beta)+
        v_2\frac{\partial \log f}{\partial \beta}(\eta_t,\alpha,\beta)+
        v_3\frac{\partial \log f}{\partial \mu},
    \end{equation*}
    with $v=(v_1,v_2,v_3)$. We have when $x\rightarrow+\infty$,
    $x\frac{f'(x,\alpha,\beta)} {f(x,\alpha,\beta)} \sim K$, then using
    Equation~\eqref{eq: proof, asymptotique du log} and letting $x
    \rightarrow\infty$, we obtain $v_1=0$ and $\forall x \in \mathbb{R}$
    \begin{equation*}
        Q\left(f(x,\psi)+xf'(x,\psi)\right) = v_2\frac{\partial f}{\partial\beta}(x,\psi)
            +v_3\frac{\partial f}{\partial\mu}(x,\psi).
    \end{equation*}
    Now multiplying both sides of the previous equation by $e^{itx}$ and
    integrating on $\mathbb{R}$, we recognize the characteristic
    function of a stable distribution or its derivatives and obtain for
    $t\in \mathbb{R}$,
    \begin{equation*}
        -Q t\frac{\partial\varphi_\psi(t)}{\partial t}=v_2\frac{\partial
        \varphi_\psi(t)}{\partial\beta} +v_3\frac{\partial \varphi_\psi(t)}{\partial\mu}.
    \end{equation*}
    Then, for $t>0$, it follows that
    \begin{equation*}
        Q\left(\alpha t^\alpha -i\beta\tan\frac{\alpha\pi}{2}(\alpha t^\alpha-t)\right)
        =v_2 i\tan\frac{\alpha\pi}{2}(t^\alpha -t)+v_3 it.
    \end{equation*}
    Therefore, we have $Q=0$ and then $v_2=v_3=0$. Finally, with
    Assumption \textbf{A\ref{hyp: MLE last}} we obtain $u=0_m$.

    We have, if $(u',v')\ \frac{\partial l_t}{\partial\tau}(\tau_0) =
    0$ then $(u',v')'= 0_{m+3}$ and the matrix $E\left[\frac{\partial l_t}
    {\partial\tau}(\tau_0)\frac{\partial
    l_t}{\partial\tau'}(\tau_0)\right]$ is positive.
    Finally, using the central limit theorem for martingale differences
    and the Wold-Cramér Lemma we obtain~\eqref{eq:convergence en loi  TCL}
    with $\Lambda = V\left[\frac{\partial l_t}{\partial\tau}\right]$.
\end{proof}

\begin{lem}
    \label{lem:normalité asymptotique 2}
    Under the assumptions of Theorem~\ref{thm: MLE, consistency},
    for any compact subset $\Gamma^*$ in the interior of $\Gamma$,
    we have
%    there exists a neighborhood $V(\theta_0)$ of $\theta_0$ such that
    \begin{equation}
        \label{eq:dérivées d'ordre 3}
        E\left[\underset{\tau\in \Gamma^*}{\sup}\ \left|\frac{\partial^3 l_t}
        {\partial\tau_i\tau_j\tau_k}(\tau)\right|\right] < +\infty.
    \end{equation}
\end{lem}

\begin{proof}
    For ease of notation, the next equations are written without
    their argument $\tau=(\theta,\psi) \in \Gamma$.
    We have
    \begin{eqnarray}
        \frac{\partial^2 l_t}{\partial\theta_i\partial\theta_j} &=& \frac{1}{2}
            \left(\frac{1}{\sigma_t^2}\frac{\partial^2\sigma_t^2}{\partial\theta_i\theta_j}
            -\phi_{t,i}\phi_{t,j}\right)\left(1+\eta_t(\theta)
            \frac{\partial\log f}{\partial x}\right)
            \label{eq:dérivée seconde theta theta}\\
            &&-\frac{1}{4}\phi_{t,i}\phi_{t,j}\eta_t(\theta)\left(\frac{\partial\log f}{\partial x}
                +\eta_t(\theta)\frac{\partial^2\log f}{\partial x^2}\right),\nonumber\\
        \frac{\partial^2 l_t}{\partial\theta_i\partial\psi} &=& \frac{1}{2}
            \phi_{t,i}\eta_t(\theta)\frac{\partial^2\log f}{\partial x\partial\psi},
            \label{eq:dérivée seconde theta (alpha,beta)}\\
        \frac{\partial^3 l_t}{\partial\theta_i\partial\theta_j\partial\theta_k} &=&
        -\left(\phi_{t,i}\phi_{t,j,k}+\phi_{t,j}\phi_{t,i,k}+\phi_{t,k}\phi_{t,i,j}\right)
        \left(\frac{1}{2}+\frac{3}{4}\eta_t(\theta)\frac{\partial\log f}{\partial x} -
            \frac{1}{4}\eta_t^2(\theta)\frac{\partial^2\log f}{\partial x^2}\right)\nonumber\\
        &&+\phi_{t,i}\phi_{t,j}\phi_{t,k}
            \left(1+\frac{15}{8}\eta_t(\theta)\frac{\partial\log f}{\partial x}+
            \frac{9}{8}\eta_t^2(\theta)\frac{\partial^2\log f}{\partial x^2}+
            \frac{1}{4}\eta_t^3(\theta)\frac{\partial^3\log f}{\partial x^3}\right)
            \nonumber\\
        &&+\frac{1}{2}\phi_{t,i,j,k}\left(1+\eta_t(\theta)\frac{\partial\log f}{\partial x}\right),
        \nonumber
    \end{eqnarray}

    We show that
    \begin{equation}
        \label{eq:dérivée troisième 1}
        E\ \underset{\tau\in \Gamma}{\sup}\left|\eta_t(\theta)\frac{\partial\log
        f}{\partial x}\left(\eta_t(\theta),\psi\right)\right| < \infty.
    \end{equation}
    For any $\psi\in A\times B\times C$, the function $x \mapsto
    \left|x\frac{\partial \log f}{\partial x}(x,\psi)\right|$ is
    bounded. Then, the function $\psi\mapsto \underset{x}{\sup}\
    \left|x\frac{\partial \log f}{\partial x}(x,\psi)\right|$ is
    continuous, thus since $A\times B\times C$ is a compact set, we
    obtain~\eqref{eq:dérivée troisième 1}.

    Using the same method for $\eta_t^2(\theta)\frac{\partial^2 \log
    f}{\partial x^2}(\eta_t(\theta),\psi)$ and
    $\eta_t^3(\theta)\frac{\partial^3 \log f}{\partial
    x^3}(\eta_t(\theta),\psi)$ and using Assumption
    \textbf{A\ref{hyp: MLE 8}} we obtain that, for any compact subset
    $\Gamma^*$ in the interior of $\Gamma$
    \begin{equation*}
        E\left[\underset{\tau\in V}{\sup}\ \left|\frac{\partial^3 l_t}
            {\partial\theta_i\theta_j\theta_k}(\tau)\right|\right] < +\infty.
    \end{equation*}
    With the same reasoning and other calculations, we obtain Equation
   ~\eqref{eq:dérivées d'ordre 3}.
\end{proof}

\begin{lem}
    \label{lem:normalité asymptotique 3}
    Under the assumptions of Theorem~\ref{thm: MLE, consistency}, we
    have when $n\rightarrow +\infty$
    \begin{align}
        \left\|\frac{1}{\sqrt{n}}\sum_{t=1}^n\left\{\frac{\partial l_t}{\partial\tau}(\tau_0)-
            \frac{\partial \tilde{l}_t}{\partial\tau}(\tau_0)\right\}\right\|
            &\rightarrow 0,
            \label{eq:conditions initiales 1}\\
        \underset{\tau\in\Gamma}{\sup}\left\|
            \frac1n\sum_{t=1}^n\left\{\frac{\partial^2 l_t}{\partial\tau\partial\tau'}(\tau)-
            \frac{\partial^2 \tilde{l}_t}{\partial\tau\tau'}(\tau)\right\}\right\|
            &\rightarrow 0.
            \label{eq: conditions initiales 2}
    \end{align}
\end{lem}

\begin{proof}
    We have
    \begin{align}
        \frac{\partial l_t}{\partial\theta}(\tau_0)-
            \frac{\partial \tilde{l}_t}{\partial\theta}(\tau_0) &=
            \frac{1}{2}\tilde{\phi}_t\tilde{Z}_t-\frac{1}{2}\phi_t Z_t\nonumber\\
        &= \frac{1}{2}\left\{\tilde{\phi}_t\left(\tilde{Z}_t-Z_t\right)+
            Z_t\left(\tilde{\phi}_t-\phi_t\right)\right\}\label{eq:conditions initiales 3}.
    \end{align}
    For $\psi\in A\times B\times C$, define the function $h$ as $h(x) =
    1+x\frac{\partial\log f}{\partial x}(x,\psi)$, we have
    $\tilde{Z}_t-Z_t = h\left(\tilde{\eta}_t\right)-h(\eta_t)$. When
    $x\rightarrow\pm\infty$, we have $h'(x)=O(x^{-1})$, therefore, with
    the mean value theorem, we have
    \begin{equation*}
        \left|\tilde{Z}_t-Z_t\right| < K|\tilde{\eta}_t-\eta_t|
            < K|\epsilon_t|\left|\tilde{\sigma}_t^2-\sigma_t^2\right|.
    \end{equation*}
    On the other side, concerning the second term in ~\eqref{eq:conditions initiales 3}, we obtain
    \begin{align*}
        \left|\tilde{\phi}_t-\phi_t\right| &< \frac{1}{\tilde{\sigma}_t^2}
            \left|\frac{\partial\tilde{\sigma}_t^2}{\partial\theta}-
            \frac{\partial{\sigma}_t^2}{\partial\theta}\right|
            +\frac{\partial{\sigma}_t^2}{\partial\theta}\left|
            \frac{1}{\tilde{\sigma}_t^2}-\frac{1}{{\sigma}_t^2}\right|\\
        &< K\left|\frac{\partial\tilde{\sigma}_t^2}{\partial\theta}-
            \frac{\partial{\sigma}_t^2}{\partial\theta}\right|
            +\frac{1}{{\sigma}_t^2}\frac{\partial{\sigma}_t^2}{\partial\theta}
            \left|\frac{\sigma_t^2}{\tilde{\sigma}_t^2}-1\right|.
    \end{align*}
    Concerning the derivatives relative to the stable parameter $\psi$,
    it follows with the mean value theorem that
    \begin{equation*}
        \left|\frac{\partial \log f}{\partial\alpha}(\tilde{\eta}_t,\psi)
            -\frac{\partial \log f}{\partial\alpha}({\eta}_t,\psi)\right|
            < \underset{x\in\mathbb{R}}{\sup}\left|\frac{\partial^2\log f}
            {\partial x\partial \alpha}(x,\psi)\right| |\tilde{\eta}_t-\eta_t|.
    \end{equation*}
    The derivative of $\log f$ with respect to $\alpha$ and $x$ is bounded. Besides, we have
    \begin{equation*}
        \left|\tilde{\eta}_t -\eta_t\right| = \left|\frac{\epsilon_t}{\tilde{\sigma}_t}
        -\frac{\epsilon_t}{{\sigma}_t}\right| <
        K|\epsilon_t|\left|\tilde{\sigma}_t^2-\sigma_t^2\right|.
    \end{equation*}
    We can apply the same method for the derivatives with respect to $\beta$
    and $\mu$. Therefore, using Assumption \textbf{A\ref{hyp: MLE 9}},
    the Markov inequality, the Borel-Cantelli lemma and the Cesàro
    Lemma, we easily obtain~\eqref{eq:conditions initiales 1}.

    Using~\eqref{eq:dérivée seconde theta theta},~\eqref{eq:dérivée
    seconde theta (alpha,beta)}, the second part of Assumption
    \textbf{A\ref{hyp: MLE 9}} and the same techniques as before, we
    obtain~\eqref{eq: conditions initiales 2}.
\end{proof}

\begin{proof}[Proof of Theorem~\ref{thm: MLE, consistency}]
    From the definition of the ML estimator $\tau_n$, we have
    $\frac{\partial \tilde{I}_n}{\partial\tau}(\tau_n)=0$, writing a
    Taylor expansion, we infer
    \begin{equation*}
        0=\frac{\partial \tilde{I}_n}{\partial\tau}(\tau_n)=\frac{\partial \tilde{I}_n}
        {\partial\tau}(\tau_0)+
        \frac{\partial^2 \tilde{I}_n}{\partial\tau\partial\tau'}(\tau^*)(\tau_n-\tau_0),
    \end{equation*}
    where $\tau^*$ is between $\tau_0$ and $\tau_n$. Using another
    Taylor expansion, Lemma~\ref{lem:normalité asymptotique 2}, the
    almost sure convergence of $\tau_n$ to $\tau_0$ and the ergodic
    theorem, we obtain
    \begin{equation*}
        \frac{\partial^2 I_n}{\partial\tau\partial\tau'}(\tau^*) \rightarrow
        \E\left[\frac{\partial^2 l_t}{\partial\tau\partial\tau'}(\tau_0)\right],
        \quad \mbox{a.s.}
    \end{equation*}
    Then, in view of Equations~\eqref{eq: proof, dérivées prem 1},~\eqref{eq:
    proof, dérivées prem 2},~\eqref{eq:dérivée seconde theta theta} and
   ~\eqref{eq:dérivée seconde theta (alpha,beta)}, we obtain
    \begin{equation*}
    	\E\left[\frac{\partial^2 l_t}{\partial\tau
    \partial\tau'}(\tau_0)\right] = \E\left[\frac{\partial
    l_t}{\partial\tau}(\tau_0)\frac{\partial
    l_t}{\partial\tau'}(\tau_0)\right].
    \end{equation*}
    By Lemmas~\ref{lem: normalité asymptotique 1} and~\ref{lem:normalité asymptotique 3} we can conclude and obtain~\eqref{eq: MLE, normalité asymptotique}.
\end{proof}

\subsection{Proof of the consistency in Theorem~\ref{theo: pseudo GARCH CAN}}

Let $I_n$ (resp. $l_{nt}$) be the equivalent of  $\tilde{I}_n$
(resp. $\tilde{l}_{nt}$) when an infinite past is known.
\begin{equation*}
    I_n(\tau) = \frac{1}{n}\sum_{t=1}^{n}l_{nt}(\tau),\ \mbox{with}\quad
    l_{nt}(\tau) = \frac{1}{2}\log\sigma_{nt}^2(\theta) -
    \log f\left(\frac{\epsilon_{nt}}{\sigma_{nt}(\theta)}, \psi\right).
\end{equation*}
We also need to define the equivalent of these quantities when the
processus $(\eta_{nt})_t$ is replaced by its limit in distribution
$(\eta_t)_t$. If $(\epsilon_t)_t$ is the stationary ergodic solution of
Model~\eqref{eq: pseudo GARCH, modèle sans n}, we define
$\sigma_{t}^2(\theta) = \frac{\omega}{\mathcal{B}_\theta(1)}
+\mathcal{B}_\theta^{-1}(L)\mathcal{A}_\theta(L)\epsilon_{t}^2$ and
$l_t(\tau)=\frac{1}{2}\log\sigma_{t}^2(\theta) - \log
f\left(\frac{\epsilon_{t}}{\sigma_{t}(\theta)},\psi\right)$.

Assumption \textbf{B\ref{hyp: pseudo GARCH, convergence unif eta n}} can only be used for quantities which depend on a finite number of $(\eta_{nt})_t$. Therefore, we introduce $\sigma_{nt}^{2(m)}(\theta)$, a truncated version of $\sigma_{nt}^2(\theta)$. For that, we give a vector representation of the GARCH$(p,q)$ model as in \citet{bougerol1992stationarity},
\begin{equation*}
    \underline{z}_{nt} = \underline{b}_{nt} + A_{nt} \underline{z}_{nt-1},
\end{equation*}
where
\begin{equation*}
    \underline{b}_{nt} = \left(
        \begin{array}{c}
            \omega_0\eta_{nt}^2\\
            0\\
            \vdots\\
            \omega_0\\
            0\\
            \vdots\\
        \end{array}
    \right)\in\mathbb{R}^{p+q},\
    \underline{z}_{nt} = \left(
        \begin{array}{c}
            \epsilon_{nt}^2\\
            \vdots\\
            \epsilon_{nt-q+1}^2\\
            \sigma_{nt}^2\\
            \vdots\\
            \sigma_{nt-p+1}^2\\
        \end{array}
    \right) \in\mathbb{R}^{p+q},
\end{equation*}
and
\begin{equation*}
    A_{nt} = \left(
        \begin{array}{cccccccccc}
             a_{01} \eta_{nt}^2 &  & \cdots &  & a_{0q}\eta_{nt}^2 &
             b_{01}\eta_{nt}^2 &  & \cdots &  & b_{0p}\eta_{nt}^2 \\
             1 & 0 & \cdots &  & 0 & 0 &  & \cdots &  & 0 \\
             0 & 1 & \cdots &  & 0 & 0 &  & \cdots &  & 0 \\
             \vdots & \ddots & \ddots &  & \vdots & \vdots & \ddots & \ddots &  & \vdots \\
             0 &   & \cdots & 1 & 0 & 0 &  & \cdots  & 0 & 0 \\
             a_{01} &  & \cdots &  & a_{0q} & b_{01} &  & \cdots &  & b_{0p} \\
             0 &  & \cdots &  & 0 & 1 & 0 & \cdots &  & 0 \\
             0 &  & \cdots &  & 0 & 0 & 1 & \cdots &  & 0 \\
             \vdots & \ddots & \ddots &  & \vdots & \vdots & \ddots & \ddots &  & \cdots \\
             0 &  & \cdots & 0 & 0 & 0 &  & \cdots & 1 & 0 \\
        \end{array}
             \right).
\end{equation*}
We also define $\underline{z}_t,\ \underline{b}_t$ and $A_t$, the
counterparts of $\underline{z}_{nt},\ \underline{b}_{nt}$ and
$A_{nt}$ where $\eta_{nt}$ is replaced by the iid sequence $(\eta_t)_t$
defined in~\eqref{eq: pseudo GARCH, modèle sans n}. Note that
$\gamma_{n}$ is the top Lyapunov exponent associated to the sequence
$(A_{nt})_{t\in\mathbb{Z}}$. Now, we prove that Assumption
\textbf{B\ref{hyp: pseudo GARCH, stationnarité}} implies that
$\gamma_n$ is inferior to zero.

With Lemma~\ref{lem: pseudo GARCH Lyapunov} below, we obtain for any
$n\in\mathbb{N}$,
\begin{equation}
    \label{eq: pseudo GARCH AN, ecriture infine z nt}
    \underline{z}_{nt} = \underline{b}_{nt} + \sum_{k=1}^{+\infty}
    A_{nt}A_{nt-1}\ldots A_{nt-k+1}\underline{b}_{nt-k}.
\end{equation}
We define the truncated version of $\underline{z}_{nt}$. For any
$m\in\mathbb{N}$,
\begin{equation}
    \label{eq: pseudo GARCH, définition z (m)}
    \underline{z}_{nt}^{(m)} = \underline{b}_{nt} + \sum_{k=1}^{m}
    A_{nt}A_{nt-1}\ldots A_{nt-k+1}\underline{b}_{nt-k},\
    \forall t\in\mathbb{Z},\ \forall n\in\mathbb{N}.
\end{equation}
In particular, if $X(k)$ is the $k^{\mbox{th}}$ element of the
vector $X$, we define a truncated version of $\sigma_{nt}^2$, that
is
\begin{equation}
    \label{eq: pseudo GARCH, définition sigma (m)}
    \underline{z}_{nt}^{(m)}(q+1) = \underline{b}_{nt}(q+1) + \left(\sum_{k=1}^{m}
        A_{nt}A_{nt-1}\ldots A_{nt-k+1}\underline{b}_{nt-k}\right)(q+1).
\end{equation}
The quantity $\underline{z}_{nt}^{(m)}(q+1)$ depends only on
$\left\{\eta_{nt-1},\ldots,\eta_{nt-m}\right\}$.

Then we define $\sigma_{nt}^{2(m)}(\theta)$ for any
$\theta\in\Theta$. For this purpose, we introduce another vector
representation of the model,
\begin{equation*}
    \underline{\sigma}_{nt}^2(\theta) = \underline{c}_{nt}(\theta) +
    B \underline{\sigma}_{nt-1}^2(\theta),
\end{equation*}
where
\begin{equation*}
    \underline{\sigma}_{nt}^2(\theta) =
    \left(\begin{array}{c}
        \sigma_{nt}^2(\theta)\\
        \sigma_{nt-1}^2(\theta)\\
        \vdots\\
        \sigma_{nt-p+1}^2(\theta)\\
    \end{array}\right),\;
    \underline{c}_{nt}(\theta) =
    \left(\begin{array}{c}
        \omega + \sum_{i=1}^q a_i\epsilon_{nt-i}^2\\
        0\\
        \vdots\\
        0\\
    \end{array}\right),\;
    B=
    \left(\begin{array}{cccc}
        b_1&b_2&\cdots &b_p\\
        1&0&\cdots&0\\
        \vdots&&&\\
        0&\cdots&1&0\\
    \end{array}\right).
\end{equation*}
By Assumption \textbf{B\ref{hyp: pseudo GARCH, stationnarité}}, we
have $\underset{\theta\in\Theta}{\sup}\ \rho(B) < 1$, where
$\rho(B)$ is the spectral radius of the matrix $B$, and thus, for any
$\theta\in\Theta$,
\begin{equation}
    \label{eq: pseudo GARCH, développement sigma nt}
    \underline{\sigma}_{nt}^2(\theta) = \sum_{k=0}^{+\infty} B^k
    \underline{c}_{nt-k}(\theta).
\end{equation}
We define for $m\in\mathbb{N}$,
\begin{equation}
    \label{eq: pseudo GARCH, définition sigma m}
    \sigma_{nt}^{2(m)}(\theta) = \sum_{k=0}^m B^k(1,1)
    \underline{c}_{nt-k}^{(m)}(\theta)(1),\
    \mbox{with}\
    \underline{c}_{nt}^{(m)}(\theta)(1) =\omega+\sum_{i=1}^q
    a_i\underline{z}_{nt-i}^{(m)}(q+1) \eta_{nt-i}^2.
\end{equation}
As for $\underline{z}_{nt}^{(m)}(q+1)$, the quantity $\sigma_{nt}^{2(m)}(\theta)$
depends on a finite number of $\eta_{nt'}$, but since every
$\underline{z}_{nt-i}^{(m)}(q+1)$ depends on several $\eta_{nt'}$,
$\sigma_{nt}^{2(m)}(\theta)$ depends on more than $m$ variables
$\eta_{nt'}$. To be precise, $\sigma_{nt}^{2(m)}(\theta)$ depends on
$\left\{\eta_{nt-1},\ldots,\eta_{nt-2m+q}\right\}$. Define also
$l_{nt}^{(m)}(\tau) = \frac{1}{2}\log\sigma_{nt}^{2(m)}(\theta) -
\log f\left(\frac{\epsilon_{nt}}{\sigma_{nt}^{2(m)}(\theta)},
\psi\right)$.

\begin{lem}
    Under the assumptions of Theorem~\ref{theo: pseudo GARCH CAN}, there
    exists $N\in\mathbb{N}$ such that
    \begin{equation}
        \label{eq: psuedo GARCH lyapunov}
        \forall n\geq N,\ \gamma_n < 0.
    \end{equation}
    Besides, there exist $k_0\in\mathbb{N}$ and $N\in\mathbb{N}$ such
    that
    \begin{equation}
        \label{eq: pseudo GARCH delta'}
        \chi' = \underset{n\geq N}{\sup}\ \E\left[\left\|
        A_{nk_0}A_{nk_0-1} \cdots A_{n1}\right\|^s \right] <1.
    \end{equation}
    \label{lem: pseudo GARCH Lyapunov}
\end{lem}

\begin{proof}[Proof of Lemma~\ref{lem: pseudo GARCH Lyapunov}]
    With Assumption \textbf{B\ref{hyp: pseudo GARCH, stationnarité}},
    using the norm $\left\|A\right\| = \sum |a_{ij}|$, which is a
    multiplicative norm and with Lemma 2.3 in \citet{francq2010garch},
    we have the existence of $k_0\in\mathbb{N}$ and of $s>0$ such that
    \begin{equation*}
        \chi = \E\left[\left\|A_{k_0}A_{k_0-1}\cdots A_{1}\right\|^s\right] < 1.
    \end{equation*}
    Now for $n\in\mathbb{N}$, writing $A_{nt} = A(\eta_{nt})$ to
    emphasize the fact that $A_{nt}$ only depends on $\eta_{nt}$, we
    have for $s>0$
    \begin{equation*}
        \E\left[\left\|A_{nk_0}A_{nk_0-1}\cdots A_{1n}\right\|^s\right] =
        \int_{\mathbb{R}^{k_0}} \left\|A(x_1)\cdots A(x_{k_0})\right\|^s
        f_n(x_1)\cdots f_n(x_{k_0})dx_1\cdots d{x_{k_0}}.
    \end{equation*}
    For $\eps>0$, by Assumption \textbf{B\ref{hyp: pseudo GARCH,
    convergence unif eta n}}, we have the existence of $N\in\mathbb{N}$,
    such that
    \begin{equation*}
        \forall n\geq N,\ \forall x\in\mathbb{R},\ f_n(x) \leq G(x),\
        \mbox{where}\
        G(x) = f(x) +\frac{\eps}{(1+|x|)^\delta}.
    \end{equation*}
    Then, the function $A$ is such that $\forall x\in\mathbb{R},\ 0< \|A(x)\|
    < K x^2$ and therefore
    \begin{equation*}
        \left\|A(x_1)\cdots A(x_{k_0})\right\|^s f_n(x_1)\cdots f_n(x_{k_0}) \leq
        K\prod_{i=1}^{k_0} |x_i|^{2s}G(x_i),
    \end{equation*}
    Using the asymptotic expansion~\eqref{eq:tail behavior 1} and
    choosing $0<s<\min\left(\frac{\delta-1}{2},\frac{\alpha}{2}\right)$, we infer
    \begin{equation*}
        \int_{\mathbb{R}^{k_0}}
        K\prod_{i=1}^{k_0} |x_i|^{2s} G(x_i)
        \prod_{i=1}^{k_0} dx_i < +\infty,
    \end{equation*}
    Thus, since $f_n$ simply converges to $f(.,\psi_0)$, using the
    dominated convergence theorem, we obtain
    \begin{equation*}
        \underset{n\rightarrow+\infty}{\lim}\
        \E\left[\left\|A_{nk_0}A_{nk_0-1}\cdots A_{1n}\right\|^s\right] =
        \E\left[\left\|A_{k_0}A_{k_0-1}\cdots A_{1}\right\|^s\right] = \chi < 1.
    \end{equation*}
    Therefore, for $\eps>0$, there exists $N\in\mathbb{N}$ such that, for
    $n\geq N$, we have 
	\begin{equation*}
		\E\left[\left\|A_{nk_0}A_{nk_0-1}\cdots
    A_{n1}\right\|^s\right] < 1-\eps,
	\end{equation*}	    
    and thus, $\chi' =
    \underset{n\geq N}{\sup}\ \E\left[\left\|A_{nk_0}A_{nk_0-1} \cdots
    A_{n1}\right\|^s\right] <1$ and we obtain~\eqref{eq: pseudo GARCH
    delta'}.

    Then, using Lemma 2.3 from \citet{francq2010garch}, we
    obtain~\eqref{eq: psuedo GARCH lyapunov}.
\end{proof}

\begin{lem}
    \label{lem: pseudo GARCH, moment uniforme}
    Under the assumptions of Theorem~\ref{theo: pseudo GARCH CAN}, there
    exists $s>0$ such that,
    \begin{equation}
        \label{eq: pseudo GARCH, moment uniforme}
        \underset{n\in\mathbb{N}}\sup\ \E\left|\epsilon_{nt}\right|^{2s} < +\infty,
        \quad \mbox{and}\quad \underset{n\in\mathbb{N}}\sup\ \E\sigma_{nt}^{2s} <+\infty.
    \end{equation}
\end{lem}

\begin{proof}
    For $n\geq N$, using the inequality $(x+y)^s \leq x^s+y^s$ for
    $x,y>0$ and $s<1$, Equation~\eqref{eq: pseudo GARCH AN, ecriture infine z nt},
    the fact that the norm is multiplicative, the independence of the
    processus $(\eta_{nt})_t$ and Lemma~\ref{lem: pseudo GARCH Lyapunov}, we obtain
    \begin{equation}
        \label{eq: stable garch grouimpf}
        \E\left\|\underline{z}_{nt}\right\|^s \leq \left\|
        \E\underline{b}_{n1}\right\|^s\left\{1+ \sum_{k=0}^{+\infty}
        \chi^{'k}\sum_{i=1}^{k_0} \left\{\E\left\|A_{n1}\right\|^s\right\}^i\right\}.
    \end{equation}
    Now, we prove that there exists $s>0$ such that
    $\underset{n\in\mathbb{N}}{\sup}\ \E|\eta_{nt}|^{2s} < +\infty$.
    In view of Assumption \textbf{B\ref{hyp: pseudo GARCH, convergence unif
    eta n}}, we obtain that $\E|\eta_{nt}|^{2s}$ converges toward
    $\E|\eta_t|^{2s}<+\infty$, for $s<\delta/2$. We used the dominated convergence theorem again and also the fact that for any $n\in\mathbb{N}$, we have $\E|\eta_{nt}|^{2s}<+\infty$ for a small enough $s>0$. Therefore, with~\eqref{eq: stable garch grouimpf}, it follows that
    \begin{equation*}
        \underset{n\in\mathbb{N}}{\sup}\ \E\left\|\underline{z}_{nt}\right\|^s <+\infty.
    \end{equation*}

    Now, for any $n\in\mathbb{N}$ and any $t\in\mathbb{Z}$, we have $\sigma_{nt}^2\leq\left\|\underline{z}_{nt}\right\|$ and $\epsilon_{nt}^2\leq \left\|\underline{z}_{nt}\right\|$. Consequently, we obtain ~\eqref{eq: pseudo GARCH, moment uniforme}.
\end{proof}

\begin{lem}
    \label{lem: pseudo GARCH, convergence du sigma tronqué}
    Under the assumptions of Theorem~\ref{theo: pseudo GARCH CAN}, there exists $s>0$ such that,
    \begin{equation}
        \label{eq: pseudo GARCH, convergence du sigma tronqué}
        \underset{n\in\mathbb{N}}{\sup}\ \underset{\theta\in\Theta}{\sup}\
        \E\left|\sigma_{nt}^{2s}(\theta)-\sigma_{nt}^{2(m){s}}(\theta)\right|
        < K\rho^m.
%        \underset{m\rightarrow +\infty}{\longrightarrow} 0.
    \end{equation}
\end{lem}

\begin{proof}
    We first prove that there exists $s>0$ such that
    \begin{equation}
        \label{eq: pseudo GARCH, convergence sigma tronqué en theta_0}
        \underset{n\in\mathbb{N}}{\sup}\
        \E\left|\sigma_{nt}^{2s}(\theta_0)- \underline{z}_{nt}^{(m)s}(q+1)
    %    \sigma_{nt}^{2(m){s}}(\theta)
        \right|
        <K\rho^m.
%        \underset{m\rightarrow +\infty}{\longrightarrow} 0.
    \end{equation}
    For $m\geq k_0^2$, let $\lfloor m/k_0\rfloor$ be the floor function of $m/k_0$ ($k_0$
    being defined as in Lemma~\ref{lem: pseudo GARCH Lyapunov}),
    we have
    \begin{equation*}
        \left\|\underline{z}_{nt} - \underline{z}_{nt}^{(\lfloor m/k_0\rfloor)}\right\| =
        \sum_{k=\lfloor m/k_0\rfloor +1}^{+\infty} \left\|A_{nt}\cdots A_{nt-k+1}\right\|
        \left\|\underline{b}_{nt-k}\right\|.
    \end{equation*}
    The constant $s$ can be taken such that $s<1$ and, using the
    inequality $(x+y)^s \leq x^s+y^s$ for $x,y>0$ and the independence of the processus $(\eta_{nt})_t$, we infer
    \begin{align*}
        \underset{n\geq N}{\sup}\
        \E\left[\left\|\underline{z}_{nt} - \underline{z}_{nt}^{(\lfloor m/k_0\rfloor)}
        \right\|^s\right] &\leq
        \sum_{k=\lfloor m/k_0\rfloor +1}^{+\infty}
        \underset{n\geq N}{\sup}\ \E\left[\left\|A_{nt}\cdots A_{nt-k+1}\right\|^s\right]
        \underset{n\geq N}{\sup}\ \E\left[\left\|\underline{b}_{nt-k}\right\|^s\right]\\
        &\leq \underset{n\geq N}{\sup}\ \E\left[\left\|\underline{b}_{n1}\right\|^s\right]
        \sum_{k=\lfloor m/k_0\rfloor +1}^{+\infty}
        \chi^{'k}\sum_{i=1}^{k_0}\left\{\underset{n\geq N}{\sup}\
        \E\left\|A_{n1}\right\|^s\right\}^i\\
        &\leq K\rho^m,
    \end{align*}
    defining $N\in\mathbb{N}$ and using similar arguments as in the proof of Lemma~\ref{lem: pseudo
    GARCH, moment uniforme}. With exactly the same arguments, we obtain
    for any $n\in\mathbb{N}$ the existence of $K_n>0$ and $\rho_n <1$
    such that
    \begin{equation*}
        \E\left[\left\|\underline{z}_{nt} - \underline{z}_{nt}^{(\lfloor m/k_0\rfloor)}
        \right\|^s\right]  \leq K_n\rho_n^m.
    \end{equation*}
    Thus, there exist $K>0$ and $\rho<1$ such that
    \begin{equation*}
        \underset{n\in\mathbb{N}}{\sup}\
        \E\left[\left\|\underline{z}_{nt} - \underline{z}_{nt}^{(\lfloor m/k_0\rfloor)}
        \right\|^s\right] \leq K\rho^m.
    \end{equation*}
    Then, we use the inequality $\left|\sigma^{2s}_{nt}-
    \underline{z}_{nt}^{(m)s}(q+1)\right| \leq \left|\sigma^{2}_{nt}-
    \underline{z}_{nt}^{(m)}(q+1)\right|^s\leq
    \left\|\underline{z}_{nt} - \underline{z}_{nt}^{(m)} \right\|^s$ and
    obtain Equation~\eqref{eq: pseudo GARCH, convergence sigma tronqué en
    theta_0}. We also obtain
    \begin{equation}
        \label{eq: pseudo GARCH, convergence des sigma le tout à la puissance s}
        \underset{n\in\mathbb{N}}{\sup}\
        \E\left[\left|\sigma^{2}_{nt}-\underline{z}_{nt}^{(m)}(q+1)\right|^s\right] \leq K\rho^m.
    \end{equation}

    We now prove the inequality~\eqref{eq: pseudo GARCH, convergence du sigma
    tronqué}. We remark that for any $m\in\mathbb{N}$ and for any
    $\theta\in\Theta$, we have $\sigma_{nt}^{2(m)}(\theta)
    \leq\sigma_{nt}^{2}(\theta)$. Then, we have
    \begin{eqnarray}
        &\underset{n\in\mathbb{N}}{\sup}&\!\! \underset{\theta\in\Theta}{\sup}\
        \E\left[\left|\sigma_{nt}^{2s}(\theta)-\sigma_{nt}^{2(m)s}(\theta)\right|\right]
        \leq
        \sum_{k=m+1}^{+\infty} \underset{\theta\in\Theta}{\sup}\ B^k(1,1)^s
        \underset{n\in\mathbb{N}}{\sup}\
        \E\left[\underset{\theta\in\Theta}{\sup}\ \underline{c}_{nt-k}^s(\theta)(1)\right]
        \label{eq: pseudo GARCH, convergence sigma m theta}\\
        &&+\sum_{k=0}^m \underset{\theta\in\Theta}{\sup}\ B^k(1,1)^s
        \sum_{i=1}^q \underset{\theta\in\Theta}{\sup}\ a_i^s\underset{n\in\mathbb{N}}{\sup}\
        \E\eta_{nt-i-k}^{2s}\underset{n\in\mathbb{N}}{\sup}\
        \E\left[\left(\sigma_{nt-i-k}^2 - \underline{z}_{nt-i-k}^{(m)}(q+1)\right)^s\right].        \nonumber
    \end{eqnarray}
    In view of the second part of Assumption \textbf{B\ref{hyp: pseudo GARCH, stationnarité}}, we have $\underset{\theta\in\Theta}{\sup}\ \rho(B)<1$. Then,
    using Lemma~\ref{lem: pseudo GARCH, moment uniforme}, we obtain
    \begin{equation*}
        \sum_{k=m+1}^{+\infty} \underset{\theta\in\Theta}{\sup}\ B^k(1,1)^s
        \underset{n\in\mathbb{N}}{\sup}\
        \E\left[\underset{\theta\in\Theta}{\sup}\ \underline{c}_{nt-k}^s(\theta)(1)\right]
%        \sum_{k=m+1}^{+\infty} B^k(1,1)^s\underset{n\in\mathbb{N}}{\sup}\
%        \E\left[\underline{c}_{nt-k}^s(\theta)(1)\right]
        \leq K\rho^m.
    \end{equation*}
    Now for the second part of~\eqref{eq: pseudo GARCH, convergence
    sigma m theta}, Equation~\eqref{eq: pseudo GARCH, convergence des sigma le tout à la puissance s} and the fact that $\rho(B) <1$ yield
    \begin{equation*}
        \sum_{k=0}^m \underset{\theta\in\Theta}{\sup}\ B^k(1,1)^s
        \sum_{i=1}^q \underset{\theta\in\Theta}{\sup}\ a_i^s\underset{n\in\mathbb{N}}{\sup}\
        \E\eta_{nt-i-k}^{2s}\underset{n\in\mathbb{N}}{\sup}\
        \E\left[\left|\sigma_{nt-i-k}^2 - \underline{z}_{nt-i-k}^{2(m)}(q+1)\right|^s\right]
%        \sum_{k=0}^m B^k(1,1)^s\sum_{i=1}^q a_i^s\underset{n\in\mathbb{N}}{\sup}\
%        \E\eta_{nt-i-k}^{2s}\underset{n\in\mathbb{N}}{\sup}\
%        \E\left[\left(\sigma_{nt-i-k}^2 - \sigma_{nt-i-k}^{2(m)}\right)^s\right]
        \leq K\rho^m.
    \end{equation*}
    Finally, having treated the two terms of the right hand of~\eqref{eq:
    pseudo GARCH, convergence sigma m theta}, we obtain~\eqref{eq:
    pseudo GARCH, convergence du sigma tronqué}.
\end{proof}

\begin{lem}
    \label{lem: pseudo GARCH, convergence des moments}
    Under the assumptions of Theorem~\ref{theo: pseudo GARCH CAN}, we
    have for any $d\in\mathbb{N}$ and for any subset $V\subset\Gamma$
    \begin{align}
        &\E\left[\left|\underset{\tau\in V}{\inf}\ l_{t}(\tau)\right|^d\right] < +\infty,\label{eq: lemme espérance, result1}\\
        &\E\left[\left(\underset{\tau\in V}{\inf}\ l_{nt}(\tau)\right)^d\right]
            \underset{n\rightarrow +\infty}{\longrightarrow}
            \E\left[\left(\underset{\tau\in V}{\inf}\ l_{t}(\tau)\right)^d\right],\label{eq: lemme espérance, result}\\
        &\E\left[\left|\underset{\tau\in V}{\inf}\ l_{nt}(\tau)\right|^d\right]
            \underset{n\rightarrow +\infty}{\longrightarrow}
            \E\left[\left|\underset{\tau\in V}{\inf}\ l_{t}(\tau)\right|^d\right]\label{eq: lemme espérance, result, abs}.
    \end{align}
%    \begin{equation}
%        \label{eq: lemme espérance, result1}
%        \E\left[\left|\underset{\tau\in V}{\inf}\ l_{t}(\tau)\right|^d\right] < +\infty,
%    \end{equation}
%    and
%    \begin{equation}
%        \label{eq: lemme espérance, result}
%        \E\left[\left(\underset{\tau\in V}{\inf}\ l_{nt}(\tau)\right)^d\right]
%            \underset{n\rightarrow +\infty}{\longrightarrow}
%            \E\left[\left(\underset{\tau\in V}{\inf}\ l_{t}(\tau)\right)^d\right]
%    \end{equation}
\end{lem}

\begin{proof}
We prove~\eqref{eq: lemme espérance, result} in the case $d=1$. The
other cases and~\eqref{eq: lemme espérance, result, abs} can be
obtained with similar arguments. We will prove the following
intermediate results. For any subset $V\subset\Gamma$
    \renewcommand{\labelenumi}{(\roman{enumi})}
    \begin{enumerate}
        \item \label{step: proof, 1} $\underset{n\in\mathbb{N}}{\sup}\
            \E\left|\underset{\tau\in V}{\inf}\ l_{nt}(\tau) -
            \underset{\tau\in V}{\inf}\
            l_{nt}^{(m)}(\tau)\right| < K\rho^m$.
            %\rightarrow 0$, when
%            $m\rightarrow +\infty$.
        \item  $\E\left|\underset{\tau\in V}{\inf}\ l_{t}(\tau) -
            \underset{\tau\in V}{\inf}\ l_{t}^{(m)}(\tau)\right|
            < K\rho^m$.
%            \rightarrow 0$, when $m\rightarrow +\infty$.
        \item For any $m>0$, $\E \underset{\tau\in V}{\inf}\
            l_{nt}^{(m)}(\tau) \rightarrow \E \underset{\tau\in
            V}{\inf}\ l_t^{(m)}(\tau)$, when
            $n\rightarrow+\infty$.
    \end{enumerate}
    We have for any $\theta\in\Theta$, $\sigma_{nt}^2(\theta) \geq
    \omega$. Since $\Theta$ is a compact set, there exists
    $\underline{\omega}> 0$ such that, $\forall\theta\in\Theta,\ \forall
    t\in\mathbb{Z},\ \forall n\in\mathbb{N},\ \sigma_{nt}^2(\theta) \geq
    \underline{\omega}$ and
    $\forall\theta\in\Theta,\ \forall t\in\mathbb{Z},\ \forall
    n\in\mathbb{N},\ \sigma_{nt}^{2(m)}(\theta) \geq\underline{\omega}$.
    From Lemma~\ref{lem: pseudo GARCH, convergence du sigma tronqué} and
    using the mean value theorem, it follows that
    \begin{align}
        \underset{n\in\mathbb{N}}{\sup}\ \underset{\theta\in\Theta}{\sup}\
        \E\left|
        \log\sigma_{nt}^2(\theta)-\log\sigma_{nt}^{2(m)}(\theta)\right| &
        \leq K \underset{n\in\mathbb{N}}{\sup}\ \underset{\theta\in\Theta}{\sup}\
        \E \left|\sigma_{nt}^{2s}(\theta)- \sigma_{nt}^{2(m)s}(\theta) \right| < K\rho^m.
%        \\
%        &\rightarrow 0,\ \mbox{when} \ m\rightarrow+\infty.
        \label{eq: pseudo GARCH, controle log}
    \end{align}
    For $\theta\in\Theta$, let $a_{nt}(\theta) =
    \frac{\sigma_{nt}(\theta_0)}{\sigma_{nt}(\theta)}$ and let
    $a_{nt}^{(m)}(\theta)=
    \sqrt{\frac{\sigma_{nt}^{2(m)}(\theta_0)}{\sigma_{nt}^{2(m)}(\theta)}}$.
    We have for $s'>0$
    \begin{equation*}
        \left|a_{nt}^{2s'}(\theta)-a_{nt}^{(m)2s'}(\theta)\right| \leq
        \sigma_{nt}^{2s'}(\theta_0) \left|\frac{1}{\sigma_{nt}^{2s'}(\theta)} -
        \frac{1}{\sigma_{nt}^{2(m)s'}(\theta)}\right| +
        \frac{1}{\sigma_{nt}^{2(m)s'}(\theta)} \left|\sigma_{nt}^{2s'}(\theta_0) -
        \sigma_{nt}^{2(m)s'}(\theta_0)\right|,
    \end{equation*}
    Setting $s'=s/2$ and using the Cauchy-Schwarz inequality and the
    results of Lemmas~\ref{lem: pseudo GARCH, moment uniforme} and
   ~\ref{lem: pseudo GARCH, convergence du sigma tronqué}, we obtain
    \begin{equation*}
        \underset{n\in\mathbb{N}}{\sup}\ \underset{\theta\in\Theta}{\sup}\
        \E\left|a_{nt}^{s'}(\theta) -
        a_{nt}^{(m)s'}(\theta) \right| < K\rho^m.
        %\rightarrow 0,\ \mbox{when}\ m\rightarrow+\infty.
    \end{equation*}
    Then, using the independence between $\sigma_{nt}^2(\theta)$ (or
    $\sigma_{nt}^{2(m)}$) and $\eta_{nt}$ and Assumption
    \textbf{B\ref{hyp: pseudo GARCH, convergence unif eta n}}, we obtain
    for any $\theta\in\Theta$
    \begin{equation*}
        \underset{n\in\mathbb{N}}{\sup}\ \underset{\theta\in\Theta}{\sup}\
        \E\left[ |\eta_{nt}|^{s'}
        \left|a_{nt}^{s'}(\theta) -a_{nt}^{(m)s'}(\theta)\right|\right] < K\rho^m.
%        \rightarrow 0,\ \mbox{when}\ m\rightarrow+\infty.
    \end{equation*}
    Defining the function $F_\psi(x) = \log f(x^{1/{s'}},\psi)$, we
    have, if $\eta_{nt}>0$
    \begin{equation*}
        \left|\log f\left(a_{nt}(\theta)\eta_{nt},\psi\right) -
        \log f\left(a_{nt}^{(m)}(\theta)\eta_{nt},\psi\right)\right| =
        \left| F_\psi\left(a_{nt}^{s'}(\theta)|\eta_{nt}|^{s'}\right) -
        F_\psi\left(a_{nt}^{(m)s'}(\theta)|\eta_{nt}|^{s'}\right) \right|.
    \end{equation*}
    The derivative of $F$ is such that $\frac{\partial F(x)}{\partial x}
    = x^{1/s'-1} \frac{f'(x^{1/s'},\psi)}{f(x^{1/s'}, \psi)}$. We have,
    when $x\rightarrow +\infty$, $\frac{\partial F_\psi}{\partial x}\sim
    1/x$. Therefore if we take $s'<1$ we obtain that $\frac{\partial
    F_\psi}{\partial x}$ is bounded. Then since $\Gamma$ is a compact
    set and since $\psi \mapsto \underset{x}{\sup}\ \frac{\partial
    F_\psi}{\partial x}(x)$ is continuous, we obtain
    $\underset{\tau\in\Gamma}{\sup}\ \underset{x}{\sup}\ \frac{\partial
    F_\psi}{\partial x}(x) < +\infty$. In view of the mean value theorem, it can be seen that
    \begin{equation*}
        \underset{\tau\in\Gamma}{\sup}\
        \left|\log f\left(a_{nt}(\theta)\eta_{nt},\psi\right) -
        \log f\left(a_{nt}^{(m)}(\theta)\eta_{nt},\psi\right)\right|
        \leq K |\eta_{nt}|^{s'} \underset{\theta\in \Theta}{\sup}\
        \left|a_{nt}^{s'}(\theta) -a_{nt}^{(m)s'}(\theta)\right|,
    \end{equation*}
    and finally
    \begin{equation}
        \label{eq: pseudo GARCH, convergence log f}
        \underset{n\in\mathbb{N}}{\sup}\ \underset{\tau\in\Gamma}{\sup}\ \E
        \left|\log f\left(a_{nt}(\theta)\eta_{nt},\psi\right) -
        \log f\left(a_{nt}^{(m)}(\theta)\eta_{nt},\psi\right)\right|<K\rho^m.
%        \rightarrow 0,\ \mbox{when}\ m\rightarrow +\infty.
    \end{equation}
    Using Equations~\eqref{eq: pseudo GARCH, controle log} and~\eqref{eq:
    pseudo GARCH, convergence log f}, we obtain
    \begin{equation}
        \label{eq: Lemme cv esp temp}
        \underset{n\in\mathbb{N}}{\sup}\ \underset{\tau\in\Gamma}{\sup}\ \E
        \left|l_{nt}(\tau)-l_{nt}^{(m)}(\tau)\right| <K\rho^m.
        %\rightarrow 0,\
%        \mbox{when}\ m\rightarrow+\infty.
    \end{equation}
    Now for $m\in\mathbb{N}$, for $K_1>0$ and for $|\rho_1|<1$, for any $n\in\mathbb{N}$, there exists $\tilde{\tau}_{m,n}\in\Gamma$ such that $l_{nt}\left(\tilde{\tau}_{m,n}\right) - \underset{\tau\in V}{\inf}\ l_{nt}(\tau) < K_1\rho_1^m$ and there exists $\hat{\tau}_{m,n}\in\Gamma$ such that $l_{nt}^{(m)}\left(\hat{\tau}_{m,n}\right) - \underset{\tau\in V}{\inf}\ l_{nt}^{(m)}(\tau) < K_1\rho_1^m$. Now if $l_{nt}\left(\hat{\tau}_{m,n}\right) \leq l_{nt}\left(\tilde{\tau}_{m,n}\right)$, we have
    \begin{align*}
        \left|\underset{\tau\in V}{\inf}\ l_{nt}(\tau) - \underset{\tau\in V}{\inf}\ l_{nt}^{(m)}(\tau)\right| &\leq  \left|\underset{\tau\in V}{\inf}\ l_{nt}(\tau) - l_{nt}\left(\hat{\tau}_{m,n}\right)\right| + \left|l_{nt}\left(\hat{\tau}_{m,n}\right) - \underset{\tau\in V}{\inf}\ l_{nt}^{(m)}(\tau)\right|\\
        &\leq \left|\underset{\tau\in V}{\inf}\ l_{nt}(\tau) - l_{nt}\left(\tilde{\tau}_{m,n}\right)\right| + K_1\rho_1^m \leq K\rho^m.
    \end{align*}
    Or if $l_{nt}^{(m)}\left(\tilde{\tau}_{m,n}\right) \leq l_{nt}^{(m)}\left(\hat{\tau}_{m,n}\right)$, we have
    \begin{align*}
        \left|\underset{\tau\in V}{\inf}\ l_{nt}(\tau) - \underset{\tau\in V}{\inf}\ l_{nt}^{(m)}(\tau)\right| &\leq  \left|\underset{\tau\in V}{\inf}\ l_{nt}(\tau) - l_{nt}\left(\tilde{\tau}_{m,n}\right)\right| + \left|l_{nt}\left(\tilde{\tau}_{m,n}\right) - \underset{\tau\in V}{\inf}\ l_{nt}^{(m)}(\tau)\right|\\
        &\leq K_1\rho_1^m + \left|l_{nt}^{(m)}\left(\hat{\tau}_{m,n}\right) - \underset{\tau\in V}{\inf}\ l_{nt}^{(m)}(\tau) \right| \leq K\rho^m.
    \end{align*}
    Now, if $l_{nt}\left(\hat{\tau}_{m,n}\right) > l_{nt}\left(\tilde{\tau}_{m,n}\right)$ and $l_{nt}^{(m)}\left(\tilde{\tau}_{m,n}\right) > l_{nt}^{(m)}\left(\hat{\tau}_{m,n}\right)$, we have
    \begin{eqnarray}
        \left|\underset{\tau\in V}{\inf}\ l_{nt}(\tau) - \underset{\tau\in V}{\inf}\ l_{nt}^{(m)}(\tau)\right| &\leq&  \left|\underset{\tau\in V}{\inf}\ l_{nt}(\tau) - l_{nt}\left(\tilde{\tau}_{m,n}\right)\right| + \left|l_{nt}\left(\tilde{\tau}_{m,n}\right) - l_{nt}^{(m)}\left(\hat{\tau}_{m,n}\right)\right|\nonumber\\
        &&+ \left|l_{nt}^{(m)}\left(\hat{\tau}_{m,n}\right) - \underset{\tau\in V}{\inf}\ l_{nt}^{(m)}(\tau) \right|.
        \label{eq: bidouille de inf}
    \end{eqnarray}
    We have
    \begin{equation*}
        l_{nt}\left(\tilde{\tau}_{m,n}\right) - l_{nt}^{(m)}\left(\tilde{\tau}_{m,n}\right) \leq l_{nt}\left(\tilde{\tau}_{m,n}\right) - l_{nt}^{(m)}\left(\hat{\tau}_{m,n}\right) \leq l_{nt}\left(\hat{\tau}_{m,n}\right) - l_{nt}^{(m)}\left(\hat{\tau}_{m,n}\right),
    \end{equation*}
    and thus, with~\eqref{eq: Lemme cv esp temp} we obtain $\E \left|l_{nt}\left(\tilde{\tau}_{m,n}\right) - l_{nt}^{(m)}\left(\hat{\tau}_{m,n}\right)\right| < K\rho^m$. Finally, with Equation~\eqref{eq: bidouille de inf} we obtain (i), the step (ii) can be obtained in the exact same way.

    [step (iii)] We have, for $m\in\mathbb{N}^*$ and $\tau\in\Gamma$,
    $l_{nt}^{(m)}(\tau) = \frac{1}{2}\log\sigma_{nt}^{2(m)}(\theta)
    -\log f\left(a_{nt}^{(m)}(\theta)\eta_{nt},\psi\right)$. The
    quantity $\sigma_{nt}^{2(m)}(\theta)$ depends on a finite number of
    $\eta_{nt}$. More precisely $\sigma_{nt}^{2(m)}(\theta)$ is a
    function of $\left\{\eta_{nt-k},\ k\in\left\{1,\ldots,
    2m+q\right\}\right\}$. Now, from~\eqref{eq: pseudo GARCH,
    développement sigma nt} we obtain that the expression of
    $\sigma_{nt}^{2(m)}(\theta)$ contains only products of powers of
    $\eta_{nt'}$. Therefore, since $\Theta$ is a compact set, there
    exist $M>0$ and $(r_1,\ldots,r_{2m+q})\in\mathbb{N}^{2m+q}$ such that
    \begin{equation}
        \label{eq: gloupps}
        \forall \theta\in\Theta,\
        \underline{\omega}\leq \sigma_{nt}^{2(m)}(\theta)\leq
        K\max(M,\eta_{nt-1}^2)^{r_1}\ldots\max(M,\eta_{nt-2m-q}^2)^{r_{2m+q}}.
    \end{equation}
    Using the same arguments, it follows that
    \begin{equation*}
        \forall \theta\in\Theta,\ a_{nt}^{(m)}(\theta) \leq K
        \max(M,\eta_{nt-1}^2)^{s_1} \ldots \max(M,\eta_{nt-2m+q}^2)^{s_{2m+q}}.
    \end{equation*}
    Then, with the asymptotic expansion~\eqref{eq:tail behavior 1}, we
    have $\forall x\in\mathbb{R},\ f(x,\psi) \geq K x^{-\alpha-1}$ and
    $\forall\psi\in A\times B\times C,\ \forall x\in\mathbb{R},\
    f(x,\psi) < K$. Therefore, there exist $({s}_0, \ldots, {s}_{2m+q})$
    such that
    \begin{equation}
        \label{eq: lemme cv esp, garg}
        \forall \tau\in\Gamma,\
        K\max(M,\eta_{nt}^2)^{{s}_0} \ldots\max(M,\eta_{nt-2m-q}^2)^{{s}{_{2m+q}}}
        \leq f(a_{nt}(\theta)\eta_{nt},\psi)\leq K.
    \end{equation}
    In view of \eqref{eq: gloupps} and~\eqref{eq: lemme cv esp, garg}, we
    obtain the existence of $M>0$ and $u_i>0,\
    i\in\left\{0,\ldots,2m+q\right\}$ such that
    \begin{equation*}
        \left|\underset{\tau\in V}{\inf}\ l_{nt}^{(m)}(\tau)\right| <
        K\left(1+\sum_{i=0}^{2m+q} \left|u_i\log\left\{\max(M,\eta_{nt-i}^2)\right\}\right|\right).
    \end{equation*}
    Then, we can apply the dominated convergence theorem as we did
    before and obtain~\eqref{eq: lemme espérance, result1} and (iii).

    \vspace{15pt}

    Now, to obtain~\eqref{eq: lemme espérance, result}, we use (i), (ii) and (iii) and obtain
    \begin{align*}
        \underset{n\rightarrow+\infty}{\lim}\
        \E\left[\underset{\tau\in V}{\inf}\ l_{nt}(\tau)\right] &=
        \underset{n\rightarrow+\infty}{\lim}\
        \underset{m\rightarrow+\infty}{\lim}\
        \E\left[\underset{\tau\in V}{\inf}\ l_{nt}^{(m)}(\tau)\right]\\
        &= \underset{m\rightarrow+\infty}{\lim}\
        \underset{n\rightarrow+\infty}{\lim}\
        \E\left[\underset{\tau\in V}{\inf}\ l_{nt}^{(m)}(\tau)\right]\\
        &= \underset{m\rightarrow+\infty}{\lim}\
        \E\left[\underset{\tau\in V}{\inf}\ l_{t}^{(m)}(\tau)\right]\\
        &= \E\left[\underset{\tau\in V}{\inf}\ l_{t}(\tau)\right].
    \end{align*}
    The limits inversion can be done since the convergence in $m$ is
    uniform with respect to $n$.
\end{proof}

\begin{lem}
    \label{lem: pseudo LGN}
    Under the assumptions of Theorem~\ref{theo: pseudo GARCH CAN}, for any subset $V\subset \Gamma$, we have
    \begin{equation}
        \frac{1}{n}\sum_{t=1}^{n}\underset{\tau\in V}{\inf}\
        l_{nt}(\tau) \rightarrow \E\left[\underset{\tau\in V}{\inf}\ l_t(\tau)\right],\
        a.s.\ \mbox{when}\ n\rightarrow+\infty.
        \label{eq: pseudo, LGN}
    \end{equation}
\end{lem}

\begin{proof}
Let $X_{nt} =(\underset{\tau\in V}{\inf}\ l_{nt}(\tau))^+$ and let
$M_n=\frac{1}{n}\sum_{t=1}^{n} X_{nt}$. We also define $S_n =
M_{n^2}$ and $m_n=\E X_{nt}$. We have $\E S_n = m_{n^2}$ and $\V S_n
= \frac{1}{n^2}\sum_{h=0}^{n^2-1} \Cov\left(X_{nt},X_{nt-h}\right)$.
We now prove that there exists $M>0$ such that for any
$n\in\mathbb{N}$ we have 
\begin{equation*}
	\left|\sum_{h=0}^{n^2-1} \Cov\left(X_{nt},X_{nt-h}\right)\right|<M.
\end{equation*}
As in the proof of the
previous Lemma, we define $X_{nt}^{(m)} = (\underset{\tau\in
V}{\inf}\ l_{nt}^{(m)}(\tau))^+$ and $R_{nt}^{(m)} =
X_{nt}-X_{nt}^{(m)}$. With the step (i) of the proof of Lemma
\ref{lem: pseudo GARCH, convergence des moments}, we have for any
$m\in\mathbb{N}$
\begin{equation}
    \label{eq: reste dev}
    \underset{n\in\mathbb{N}}{\sup}\ \E\left| R_{nt}^{(m)}\right| < K\rho^m\quad \mbox{and} \quad
    \underset{n\in\mathbb{N}}{\sup}\ \E R_{nt}^{(m)2} < K\rho^m.
\end{equation}
Now, for $h\in\mathbb{N}$
\begin{equation*}
    \Cov\left(X_{nt},X_{nt-h}\right) = \Cov\left(X_{nt}^{(\lfloor h/2\rfloor)},X_{nt-h}\right)
        +\Cov\left(R_{nt}^{(\lfloor h/2\rfloor)},X_{nt-h}\right).
\end{equation*}
Using~\eqref{eq: reste dev}, we obtain
\begin{equation}
    \label{eq: covariance 2}
    \underset{n\in\mathbb{N}}{\sup}\ \left|
    \Cov\left(R_{nt}^{(\lfloor h/2\rfloor)},X_{nt-h} \right)\right|<K\rho^h.
\end{equation}
It can be seen that, for any $\tau\in\Gamma$, $l_{nt}(\tau)$ can be
written as a measurable function of $(\epsilon_{nt'})_{t' \leq t}$.
Therefore, $X_{nt-h}$ is also a measurable function of
$(\epsilon_{nt'})_{t' \leq t-h}$. Besides,
$X_{nt}^{(\lfloor h/2\rfloor)}$ is a measurable function of
$(\epsilon_{nt'})_{t'\geq t-\lfloor h/2\rfloor}$. Thus, we have for
$h\in\mathbb{N}$
\begin{equation*}
    \left|\Cov\left(X_{nt}^{(\lfloor h/2\rfloor)},X_{nt-h}\right)\right| \leq
    \alpha_{\epsilon_n}(\lfloor h/2\rfloor).
\end{equation*}
Note that $\alpha_{\epsilon_n}$ is the mixing coefficient of the
process $(\epsilon_{nt})$. Therefore, Assumption
\textbf{B\ref{hyp: pseudo GARCH, mélange}} yields
\begin{equation}
    \label{eq: covariance 3}
    \underset{n\in\mathbb{N}}{\sup}\ \left|
    \Cov\left(X_{nt}^{(\lfloor h/2\rfloor)},X_{nt-h}\right)\right| < K\rho^h.
\end{equation}
\begin{sloppypar}
Now with~\eqref{eq: covariance 2} and~\eqref{eq: covariance 3} we
obtain $\underset{n\in\mathbb{N}}{\sup}\ \left|
\Cov\left(X_{nt},X_{nt-h}\right)\right| < K\rho^h$ and consequently $\sum_{n\geq 1} \V S_n < +\infty$.
\end{sloppypar}

    Then, using the Tchebychev's Inequality, we obtain
    \begin{equation}
    \sum_{n\geq 1} \mathbb{P}\left[|S_n-m_{n^2}|>\varepsilon\right] \leq
        \frac{1}{\varepsilon^2}\sum_{n\geq1} \V S_n.
        % \frac{\V X_{n^2t}}{\varepsilon^2 {n^2}},
        \label{eq: pseudo, tchebyshev ineq}
    \end{equation}
%    Using Lemma~\ref{lem: pseudo GARCH, convergence des moments}, we
%    have, when $n\rightarrow+\infty$
%    \begin{equation*}
%    \V X_{n^2t} \rightarrow \V (\underset{\tau\in V}{\inf}\ l_{t}(\tau))^+ < +\infty.
%    \end{equation*}
    Thus, since the series of Equation~\eqref{eq: pseudo, tchebyshev ineq} is
    convergent
    we obtain the almost-sure convergence of $S_n-m_{n^2}$ to 0. We also
    have from Lemma~\ref{lem: pseudo GARCH, convergence des moments} the
    almost sure convergence of $m_n$ to $\E (\underset{\tau\in V}{\inf}\
    l_{t}(\tau))^+$, therefore
    \begin{equation*}
    S_n \rightarrow \E (\underset{\tau\in V}{\inf}\ l_{t}(\tau))^+,\ a.s.
    \end{equation*}

    We now prove that $M_n$ converges also to
    $\E\left[(\underset{\tau\in V}{\inf}\ l_{t}(\tau))^+\right]$ almost surely. Let
    $q_n=\lfloor\sqrt{n}\rfloor$ be the floor function of $\sqrt{n}$.
    Since the element of the sum $M_n$ are positives, we have
    \begin{equation*}
    \frac1n q_n^2S_{q_n} \leq M_n \leq \frac1n (q_n+1)^2S_{q_n+1}.
    \end{equation*}
    Using the fact that $\frac{q_n^2}{n}$ converges to 1, we obtain the
    $M_n\rightarrow \E (\underset{\tau\in V}{\inf}\ l_{t}(\tau))^+, \
    a.s.$ Finally, using the same method for the negative part, we can
    conclude and obtain~\eqref{eq: pseudo, LGN}.
\end{proof}

\begin{lem}
    \label{lem: pseudo GARCH consistence conditions initiales}
    Under the assumptions of Theorem~\ref{theo: pseudo GARCH CAN},
    we have
    \begin{equation}
        \underset{n\rightarrow+\infty}{\lim}\ \underset{\tau\in\Gamma}{\sup}\
        \left|I_{n}(\tau)-\tilde{I}_{n}(\tau)\right| < 0,\ a.s.
    \end{equation}
\end{lem}

\begin{proof}
Let $\underline{\tilde{\sigma}}_{nt}^2$ be the vector obtained by
replacing $\sigma_{nt-i}^2$ by $\tilde{\sigma}_{nt-i}^2$ and let
$\underline{\tilde{c}}_{nt}$ be the vector obtained by replacing
$\epsilon_{n0}^2,\ldots,\epsilon_{n1-q}^2$ by some initial values.
In view of \eqref{eq: pseudo GARCH, développement sigma nt}, we have
\begin{equation*}
    \left|\underline{\sigma}_{nt}^2-\underline{\tilde{\sigma}}^2_{nt}\right| =
    \left|\sum_{k=1}^q B^{t-k}(\underline{c}_{nk}-\underline{\tilde{c}}_{nk})
    +B^t (\underline{\sigma}_{n0}^2-\underline{\tilde{\sigma}}_{n0}^2)\right|.
\end{equation*}
Assumption \textbf{B\ref{hyp: pseudo GARCH, stationnarité}} yields $\underset{\theta\in\Theta}{\sup}\ \rho(B) < 1$, consequently, we have
\begin{equation*}
    \underset{\theta\in\Theta}{\sup}\
    \left|\underline{\sigma}_{nt}^2-\underline{\tilde{\sigma}}^2_{nt}\right| \leq
    K\rho^t\left(\max(\epsilon_{n0}^2,\ldots,\epsilon_{n1-q}^2)+
    \max(\sigma_{n0}^2,\ldots,\sigma_{n1-p}^2)+1\right).
\end{equation*}
Then, since the random variables
$\max(\epsilon_{n0}^2,\ldots,\epsilon_{n1-q}^2)$ and
$\max(\sigma_{n0}^2,\ldots,\sigma_{n1-p}^2)$ possess moments of
order $s$, by Lemma~\ref{lem: pseudo GARCH, moment uniforme}, we
can conclude as we did in the proof of Theorem~\ref{thm: MLE,
consistency}.
\end{proof}

Then, the proof can be done in the exact same way as in the proof of
Theorem~\ref{thm: MLE, consistency}. Starting from Equation
\eqref{eq: consistence MLE, CI}, we can use Lemma~\ref{lem: pseudo
LGN} to conclude.

\subsection{Proof of the asymptotic normality in Theorem~\ref{theo: pseudo GARCH CAN}}

We introduce a truncated version of the derivatives of
$\sigma_{nt}^2$. From~\eqref{eq: pseudo GARCH, développement sigma
nt}, we obtain
\begin{align*}
    \frac{\partial \sigma_{nt}^2}{\partial\omega}(\theta) &=
        \sum_{k=0}^{+\infty}B^k(1,1),\\
    \frac{\partial \sigma_{nt}^2}{\partial a_i}(\theta) &=
        \sum_{k=0}^{+\infty}B^k(1,1) \epsilon_{nt-k-i}^2,\quad
        \forall i\in \left\{1,\ldots,q\right\},\\
    \frac{\partial \sigma_{nt}^2}{\partial b_j}(\theta) &=
        \sum_{k=1}^{+\infty}\left[\sum_{i=1}^k B^{i-1} B^{(j)}B^{k-1}
        \underline{c}_{nt-k}(\theta)\right](1),\quad \forall j\in\left\{1,\ldots,p\right\}.
\end{align*}
For $m\in\mathbb{N}$, we define
\begin{align}
    \label{eq: pseudo GARCH AN definition omega}
    \left(\frac{\partial \sigma_{nt}^2}{\partial\omega}\right)^{(m)}(\theta) &=
        \sum_{k=0}^{m}B^k(1,1),\\
    \label{eq: pseudo GARCH AN definition a}
    \left(\frac{\partial \sigma_{nt}^2}{\partial a_i}\right)^{(m)}(\theta) &=
        \sum_{k=0}^{m}B^k(1,1) \underline{z}_{nt-k-i}^{(m)}(q+1)\eta_{nt-k-i}^2
        ,\quad \forall i\in \left\{1,\ldots,q\right\},\\
    \label{eq: pseudo GARCH AN definition b}
    \left(\frac{\partial \sigma_{nt}^2}{\partial b_j}\right)^{(m)}(\theta) &=
        \sum_{k=1}^{m}\left[\sum_{i=1}^k B^{i-1} B^{(j)}B^{k-1}
        \underline{c}_{nt-k}^{(m)}(\theta)\right](1),\quad \forall j\in\left\{1,\ldots,p\right\}.
\end{align}
where $B^{(j)}$ is a $p\times p$ matrix with $1$ in position $(1,j)$
and zeros elsewhere. Then, we define
\begin{equation*}
    \phi_{nt}(\theta) = \frac{1}{\sigma_{nt}^2(\theta)}
    \frac{\partial\sigma_{nt}^2}{\partial\theta}(\theta),\
    \phi_{nt}^{(m)}(\theta)=  \frac{1}{\sigma_{nt}^{2(m)}(\theta)}
    \left(\frac{\partial\sigma_{nt}^2}{\partial \theta}\right)^{(m)}(\theta),
\end{equation*}
and for $i\in\left\{1,\ldots,p+q+1\right\}$,
$\phi_{nt,i}^{(m)}(\theta)= \frac{1}{\sigma_{nt}^{2(m)}(\theta)}
\left(\frac{\partial\sigma_{nt}^2}{\partial
\theta_i}\right)^{(m)}(\theta)$.

\begin{lem}
    \label{lem: pseudo GARCH, phi nt convergence en m} Under the
    assumptions of Theorem~\ref{theo: pseudo GARCH CAN}, there exists a
    neighborhood $V(\theta_0)$ of $\theta_0$ such that, for any $\theta\in
    V(\theta_0)$ and for $(i,j,k)\in\left\{1,\ldots,p+q+1\right\}$,
    we have
    \begin{align}
        \label{eq: pseudo GARCH, phi nt convergence en m}
        \underset{n\in\mathbb{N}}{\sup}\
        \E\left|\phi_{nt,i}(\theta)-\phi_{nt,i}^{(m)}(\theta)\right|&< K\rho^m,\\
%        &\underset{m\rightarrow+\infty}{\longrightarrow} 0\\
        \label{eq: pseudo GARCH, AN, double phi n}
        \underset{n\in\mathbb{N}}{\sup}\ \E\left|\phi_{nt,i}(\theta)\phi_{nt,j}(\theta)-
        \phi_{nt,i}^{(m)}(\theta)\phi_{nt,j}^{(m)}(\theta)\right|&< K\rho^m,\\
%        &\underset{m\rightarrow+\infty}{\longrightarrow} 0\\
        \label{eq: pseudo GARCH, AN, triple phi n}
        \underset{n\in\mathbb{N}}{\sup}\ \E\left|
        \phi_{nt,i}(\theta)\phi_{nt,j}(\theta)\phi_{nt,k}(\theta)-
        \phi_{nt,i}^{(m)}(\theta)\phi_{nt,j}^{(m)}(\theta)\phi_{nt,k}^{(m)}(\theta)\right|&< K\rho^m.
%        &\underset{m\rightarrow+\infty}{\longrightarrow} 0.
    \end{align}
    And
    \begin{align}
        \label{eq: pseudo GARCH, espérance phi nt sur Theta}
        \underset{\theta\in V(\theta_0)}{\sup}\
        \underset{n\in\mathbb{N}}{\sup}\ \E\left|\phi_{nt,i}(\theta)\right| < +\infty,\quad
%        \label{eq: pseudo GARCH, espérance double phi nt sur Theta}
        \underset{\theta\in V(\theta_0)}{\sup}\
        \underset{n\in\mathbb{N}}{\sup}\ \E\left|\phi_{nt,i}(\theta) \phi_{nt,j}(\theta)\right|
        < +\infty,\\
        \label{eq: pseudo GARCH, espérance triple phi nt sur Theta}
        \underset{\theta\in V(\theta_0)}{\sup}\
        \underset{n\in\mathbb{N}}{\sup}\ \E\left|\phi_{nt,i}(\theta)
        \phi_{nt,j}(\theta) \phi_{nt,k}(\theta)\right| < +\infty.
    \end{align}
\end{lem}

\begin{proof}
    In this proof, for clarity purpose, the arguments $(\theta)$ are
    omitted ($\phi_{nt}$ stands for $\phi_{nt}(\theta)$). We have for
    $n\in\mathbb{N},\ t\in\mathbb{Z},\ \theta\in\Theta$ and
    $i\in\left\{1,\ldots,q\right\}$,
    \begin{equation}
        \label{eq: pseudo GARCH, AN, phi a}
        \left|\frac{1}{\sigma_{nt}^2}\frac{\partial\sigma_{nt}^2}{\partial a_i}-
        \frac{1}{\sigma_{nt}^{2(m)}}\left(\frac{\partial\sigma_{nt}^2}
        {\partial a_i}\right)^{(m)}\right| \leq \frac{1}{\sigma_{nt}^2}\left|
        \frac{\partial\sigma_{nt}^2}{\partial a_i} -
        \left(\frac{\partial\sigma_{nt}^2}{\partial a_i}\right)^{(m)}\right|
        +\left| \left(\frac{\partial\sigma_{nt}^2}{\partial a_i}\right)^{(m)}\right|
        \left|\frac{1}{\sigma_{nt}^2}-\frac{1}{\sigma_{nt}^{2(m)}}\right|.
    \end{equation}
    We begin by the first term of the previous equation, we have
    \begin{equation*}
        \frac{\partial\sigma_{nt}^2}{\partial a_i} -
        \left(\frac{\partial\sigma_{nt}^2}{\partial a_i}\right)^{(m)} =
        \sum_{k=m+1}^{+\infty} B^k(1,1) \epsilon_{nt-k-i}^2.
    \end{equation*}
    Then, we remark that $a_i\epsilon_{nt-k-i}^2 <
    \underline{c}_{nt-k}(1)$ and that $\sigma_{nt}^2 > \omega +
    B^k(1,1)\underline{c}_{nt-k}(1)$ and we infer
    \begin{align*}
        \frac{1}{\sigma_{nt}^2}\left|
        \frac{\partial\sigma_{nt}^2}{\partial a_i} -
        \left(\frac{\partial\sigma_{nt}^2}{\partial a_i}\right)^{(m)}\right| &\leq
        \sum_{k=m+1}^{+\infty} \frac{1}{a_i}\frac{B^k(1,1)\underline{c}_{nt-k}(1)}
        {\omega +B^k(1,1)\underline{c}_{nt-k}(1)}\\
        &\leq \sum_{k=m+1}^{+\infty} \frac{1}{a_i}
        \left\{\frac{B^k(1,1)\underline{c}_{nt-k}(1)}{\omega}\right\}^s,
    \end{align*}
    using the inequality $x/(1+x)\leq x^s$ for all $x\geq 0$. With
    By \textbf{B\ref{hyp: pseudo GARCH, interior}}, we have
    $a_{0i} >0$, thus there exists a neighborhood $V(\theta_0)$ of
    $\theta_0$ such that $\underset{\theta\in V(\theta_0)}{\inf}\ a_i >0$.
    Then, using Lemma~\ref{lem: pseudo GARCH, moment uniforme} and the
    fact that the spectral radius of $B$ is inferior to 1, it follows that
    \begin{equation}
        \label{eq: pseudo GARCH, AN, premier terme d sigma a}
        \underset{n\in\mathbb{N}}{\sup}\ \E\left[\frac{1}{\sigma_{nt}^2}\left|
        \frac{\partial\sigma_{nt}^2}{\partial a_i} -
        \left(\frac{\partial\sigma_{nt}^2}{\partial a_i}\right)^{(m)}\right|\right]\leq
        K\rho^m.
    \end{equation}
    Turning to the second term of Equation~\eqref{eq: pseudo GARCH, AN, phi
    a}, the mean value theorem applied to the function $x\mapsto {x^{-1/s}}$ yields
    \begin{equation*}
        \left|\frac{1}{\sigma_{nt}^2}-\frac{1}{\sigma_{nt}^{2(m)}}\right|\leq
        K\frac{1}{\tilde{\sigma}^{1/s+1}}\left|\sigma_{nt}^{2s}-\sigma_{nt}^{2(m)s}\right|,
    \end{equation*}
    where $\tilde{\sigma}$ is between $\sigma_{nt}^{2s}$ and
    $\sigma_{nt}^{2(m)s}$. Since $\tilde{\sigma} \geq\sigma_{nt}^{2(m)s}$, it follows that
    \begin{equation}
        \label{eq: pseudo GARCH, AN, glups}
        \left|\frac{1}{\sigma_{nt}^2}-\frac{1}{\sigma_{nt}^{2(m)}}\right|\leq
        K\frac{1}{\sigma_{nt}^{2(m)+2s}}\left|\sigma_{nt}^{2s}-\sigma_{nt}^{2(m)s}\right|.
    \end{equation}
    Then, we have
    \begin{equation*}
        a_i\left(\frac{\partial\sigma_{nt}^2}{\partial a_i}\right)^{(m)} =
        \sum_{k=0}^mB^k(1,1)a_i \epsilon_{nt-k-i}^2\leq
        \sum_{k=0}^mB^k(1,1)\underline{c}_{nt-k-i}(1) = \sigma_{nt}^{2(m)},
    \end{equation*}
    and thus, $\frac{1}{\sigma_{nt}^{2(m)}} \left(
    \frac{\partial\sigma_{nt}^2}{\partial a_i}\right)^{(m)} \leq K$.
    Now, with~\eqref{eq: pseudo GARCH, AN, glups} and using Lemma
   ~\ref{lem: pseudo GARCH, convergence du sigma tronqué}, we obtain
    \begin{align*}
        \underset{n\in\mathbb{N}}{\sup}\ \E\left[
        \left| \left(\frac{\partial\sigma_{nt}^2}{\partial a_i}\right)^{(m)}\right|
        \left|\frac{1}{\sigma_{nt}^2}-\frac{1}{\sigma_{nt}^{2(m)}}\right|\right] &\leq K
        \underset{n\in\mathbb{N}}{\sup}\ \E\left[
        \left| \frac{1}{\sigma_{nt}^{2(m)}}
        \left(\frac{\partial\sigma_{nt}^2}{\partial a_i}\right)^{(m)}\right|
        \frac{1}{\sigma_{nt}^{2(m)s}}\left|\sigma_{nt}^{2s}-\sigma_{nt}^{2(m)s}\right|\right]\\
        &\leq K\rho^m.
    \end{align*}
    Finally, in view of the previous equation and~\eqref{eq: pseudo GARCH, AN, premier terme d sigma a}, we obtain
    \begin{equation*}
        \underset{n\in\mathbb{N}}{\sup}\ \E
        \left|\frac{1}{\sigma_{nt}^2}\frac{\partial\sigma_{nt}^2}{\partial a_i}-
        \frac{1}{\sigma_{nt}^{2(m)}}\left(\frac{\partial\sigma_{nt}^2}
        {\partial a_i}\right)^{(m)}\right|<K\rho^m.
        %\underset{m\rightarrow+\infty}{\longrightarrow} 0.
    \end{equation*}
    If we adapt the proof for the derivatives with respect to $b_j$ and
    $\omega$, we obtain~\eqref{eq: pseudo GARCH, phi nt convergence en
    m}. Now for the first part of~\eqref{eq: pseudo GARCH, espérance phi
    nt sur Theta}, for any $n\in\mathbb{N}$ with already used
    arguments, we have
    \begin{align*}
        \frac{1}{\sigma_{nt}^2}\frac{\partial\sigma_{nt}^2}{\partial a_i} &\leq K,\quad
        \frac{1}{\sigma_{nt}^2}\frac{\partial\sigma_{nt}^2}{\partial \omega} \leq K\\
        \E\left[\frac{1}{\sigma_{nt}^2}\frac{\partial\sigma_{nt}^2}{\partial
        b_j}\right] &\leq \frac{1}{b_j}\sum_{k=1}^{+\infty}k\E\left\{
        \frac{B^k(1,1)\underline{c}_{nt-k}(1)}{\omega}\right\}^s \leq
        \frac{K}{b_j}.
    \end{align*}
    And the first part of~\eqref{eq: pseudo GARCH, espérance phi nt sur
    Theta} comes easily.

    \vspace{15pt}

    Turning to~\eqref{eq: pseudo GARCH, AN, double phi n}, we have for
    $(i,j)\in\left\{1,\ldots,p+q+1\right\}^2$
    \begin{equation*}
        \left|\phi_{nt,i}\phi_{nt,j}'- \phi_{nt,i}^{(m)}\phi_{nt,j}^{(m)'}\right|\leq
        \left|\phi_{nt,i}\right|\ \left|\phi_{nt,j}^{(m)}-\phi_{nt,j}^{(m)'}\right| +
        \left|\phi_{nt,i}-\phi_{nt,i}\right|\ \left|\phi_{nt,j}^{(m)'}\right|.
    \end{equation*}
    By \eqref{eq: pseudo GARCH, phi nt convergence en m} and the first
    part of~\eqref{eq: pseudo GARCH, espérance phi nt sur Theta}, we
    obtain $\underset{\theta\in V(\theta_0)}{\sup}\ \underset{n\in\mathbb{N}}{\sup}\ \left|\phi_{nt,j}^{(m)}\right| < +\infty$ and~\eqref{eq: pseudo GARCH, AN, double phi n}. All the other
    results of the lemma can be obtained with similar arguments.
\end{proof}

\begin{lem}
    \label{lem: pseudo GARCH, AN, convergence d ln en m}
    Defining
    \begin{equation*}
        \left(\frac{\partial l_{nt}}{\partial\tau}\right)^{(m)}(\tau_0) =
        \left(\begin{array}{c}
            \frac{1}{2}\phi_{nt}^{(m)}\left(1+\eta_{nt}
            \frac{\partial \log f}{\partial x}(\eta_{nt},\psi_0)\right)\\
            -\frac{\partial\log f}{\partial\psi}(\eta_{nt},\psi_0)
        \end{array}\right),
    \end{equation*}
    and under the assumptions of Theorem~\ref{theo: pseudo GARCH CAN},  we have
    \begin{equation}
        \label{eq: pseudo GARCH, AN, convergence dérivée ln en m}
        \underset{n\in\mathbb{N}}{\sup}\ \E
        \left\|\frac{\partial l_{nt}}{\partial\tau}(\tau_0)-
        \left(\frac{\partial l_{nt}}{\partial\tau}\right)^{(m)}(\tau_0)\right\| < K\rho^m.
%        \underset{m\rightarrow+\infty}{\longrightarrow} 0.
    \end{equation}
\end{lem}

\begin{proof}
    From~\eqref{eq: proof, dérivées prem 1}, we have
    \begin{equation*}
        \frac{\partial l_{nt}}{\partial\theta}(\tau_0)-
        \left(\frac{\partial l_{nt}}{\partial\theta}(\tau_0)\right)^{(m)} =
        \frac{1}{2}\left(1+\eta_{nt} \frac{\partial \log f}{\partial x}(\eta_{nt},\psi_0)\right)
        \left(\phi_{nt}-\phi_{nt}^{(m)}\right).
    \end{equation*}
    Now, since $\left(1+\eta_{nt} \frac{\partial \log f}{\partial
    x}(\eta_{nt},\psi_0)\right)$ only depends on $\eta_{nt}$, we can
    apply the dominated convergence theorem with Assumption
    \textbf{B\ref{hyp: pseudo GARCH, convergence unif eta n}} and obtain
    \begin{equation*}
        \underset{n\rightarrow+\infty}{\lim}\
        \E\left[1+\eta_{nt} \frac{\partial \log f}{\partial x}(\eta_{nt},\psi_0)\right] =
        \E\left[1+\eta_{t} \frac{\partial \log f}{\partial x}(\eta_{t},\psi_0)\right] <+\infty.
    \end{equation*}
    The last inequality has been proved in Section~\ref{sec: proof of
    the asymptotic normality MLE}. Since the function $x\mapsto
    \left(1+x \frac{\partial \log f}{\partial x}(x,\psi_0)\right)$ is
    bounded, it is now clear that
    \begin{equation*}
        \underset{n\in\mathbb{N}}{\sup}\
        \E\left[1+\eta_{nt} \frac{\partial \log f}{\partial x}(\eta_{nt},\psi_0)\right] < +\infty.
    \end{equation*}
    Then, with Lemma~\ref{lem: pseudo GARCH, phi nt convergence en m} we
    obtain
    \begin{equation*}
        \underset{n\in\mathbb{N}}{\sup}\ \E
        \left\|\frac{\partial l_{nt}}{\partial\theta}(\tau_0)-
        \left(\frac{\partial l_{nt}}{\partial\theta}(\tau_0)\right)^{(m)}\right\|<K\rho^m.
%        \underset{m\rightarrow+\infty}{\longrightarrow} 0.
    \end{equation*}
    Finally, since $\left(\frac{\partial l_{nt}}{\partial\psi}(\tau_0)
    \right)^{(m)} =\frac{\partial l_{nt}}{\partial\psi}(\tau_0)$ we
    obtain~\eqref{eq: pseudo GARCH, AN, convergence dérivée ln en m}.
\end{proof}

\begin{lem}
    \label{ref: pseudo GARCH AN convergence d lnt en n}
    Under the assumptions of Theorem~\ref{theo: pseudo GARCH CAN}, we have
    \begin{equation}
        \label{eq: pseudo GARCH, AN, convergence d lnt en n}
        \underset{n\rightarrow+\infty}{\lim}\
        \E\frac{\partial l_{nt}}{\partial\tau}(\tau_0) =
        \E\frac{\partial l_{t}}{\partial\tau}(\tau_0) \quad\mbox{and}\quad
        \underset{n\rightarrow+\infty}{\lim}\
        \E\left[\frac{\partial l_{nt}}{\partial\tau}(\tau_0)
        \frac{\partial l_{nt}}{\partial\tau'}(\tau_0)\right] =
        \E\left[\frac{\partial l_{t}}{\partial\tau}(\tau_0)
        \frac{\partial l_{t}}{\partial\tau'}(\tau_0)\right].
    \end{equation}
%    and
%    \begin{equation}
%        \label{eq: pseudo GARCH, AN, convergence double d lnt en n}
%        \underset{n\rightarrow+\infty}{\lim}\
%        \E\left[\frac{\partial l_{nt}}{\partial\tau}(\tau_0)
%        \frac{\partial l_{nt}}{\partial\tau'}(\tau_0)\right] =
%        \E\left[\frac{\partial l_{t}}{\partial\tau}(\tau_0)
%        \frac{\partial l_{t}}{\partial\tau'}(\tau_0)\right].
%    \end{equation}
\end{lem}

\begin{proof}
    To obtain the result, we first prove that, for any $m\in\mathbb{N},\ \E\phi_{nt}^{(m)}
    (\theta_0)\rightarrow \E\phi_{t}^{(m)} (\theta_0)$ when
    $n\rightarrow+\infty$. From~\eqref{eq: pseudo GARCH AN definition
    omega}, we know that $\left(
    \frac{\partial\sigma_{nt}^2}{\partial\omega}
    \right)^{(m)}(\theta_0)$ does not depend on $n$. Since
    $\sigma_{nt}^2\geq \omega_0>0$, we can apply the dominated
    convergence theorem and obtain
    \begin{equation*}
        \E\phi_{nt,1}^{(m)} (\theta_0)\rightarrow \E\phi_{t,1}^{(m)} (\theta_0),\
        \mbox{when} \ n\rightarrow+\infty.
    \end{equation*}
    Now, using the same method as in the proof of Lemma~\ref{lem: pseudo GARCH, phi
    nt convergence en m}, we infer for $i\in\left\{1,\cdots,q\right\}$
    \begin{equation*}
        \phi_{nt,1+i}^{(m)} \leq \sum_{k=0}^m\frac{1}{a_{0,i}}
        \left\{\frac{B^k(1,1)\underline{c}_{nt-k}^{(m)}(1)}{\omega_0}\right\}^s,
    \end{equation*}
    \begin{sloppypar}
\noindent    where $s$ can be chosen as small as wanted. The same arguments as in the proof of Lemma~\ref{lem: pseudo GARCH, convergence des moments} and the previous equation imply that $\phi_{nt,1+i}^{(m)}$ is a function of $\left\{\eta_{nt-k};1\leq k\leq 2m+q\right\}$ and is such that, for any $s>0$, there exist $K,M>0$ such that $\phi_{nt,1+i}^{(m)}\leq K\prod_{i=1}^{2m+q}\max\left(M, \eta_{nt-i}^s\right)$. Then, in view of the dominated convergence theorem, we obtain
    \end{sloppypar}
    \begin{equation*}
        \E\phi_{nt,1+i}^{(m)} (\theta_0)\rightarrow \E\phi_{t,1+i}^{(m)} (\theta_0),\
        \mbox{when} \ n\rightarrow+\infty.
    \end{equation*}
    Doing exactly the same for $\phi_{nt,1+q+j}^{(m)}$ (as in Lemma
   ~\ref{lem: pseudo GARCH, phi nt convergence en m}) with $j\in\left\{1,\ldots,p\right\}$, we obtain
    \begin{equation*}
        \E\phi_{nt}^{(m)} (\theta_0)\rightarrow \E\phi_{t}^{(m)} (\theta_0),\
        \mbox{when} \ n\rightarrow+\infty.
    \end{equation*}
    Since $1+\eta_{nt}\frac{\partial\log f}{\partial
    x}(\eta_{nt},\psi_0)$ only depends on $\eta_{nt}$, we easily obtain
    \begin{equation*}
        \E\left[1+\eta_{nt}\frac{\partial\log f}{\partial x}(\eta_{nt},\psi_0)\right]
        \underset{n\rightarrow+\infty}{\longrightarrow} 0.
    \end{equation*}
    Consequently,
    \begin{equation*}
        \underset{n}{\lim}\ \E\left(
        \frac{\partial l_{nt}}{\partial\theta}\right)^{(m)}
        (\theta_0) = \E\left(
        \frac{\partial l_{t}}{\partial\theta}\right)^{(m)}
        (\theta_0).
    \end{equation*}
    Now, with the asymptotic expansions~\eqref{eq:tail behavior
    3}-\eqref{eq:tail behavior 5} and with previously used arguments, we
    also obtain the convergence for the derivatives with respect to
    $\psi$. Finally, inverting the double limit $ \underset{n}{\lim}\
    \underset{m}{\lim}\ \E\left(\frac{\partial
    l_{nt}}{\partial\tau}\right)^{(m)} (\theta_0) = \underset{m}{\lim}\
    \underset{n}{\lim}\ \E\left(\frac{\partial
    l_{nt}}{\partial\tau}\right)^{(m)} (\theta_0)$, we obtain the first part of~\eqref{eq:
    pseudo GARCH, AN, convergence d lnt en n}. It is clear that
    the second part of~\eqref{eq: pseudo GARCH, AN, convergence d lnt en n} can be
    obtained with very similar arguments.
\end{proof}

\begin{lem}
    Under the assumptions of Theorem~\ref{theo: pseudo GARCH CAN},
    we have
    \begin{equation}
        \label{eq: pseudo GARCH AN TCL}
        \sqrt{n}\frac{\partial I_n}{\partial\tau}(\tau_0) \arrowloi \mathcal{N}(0,J),
    \end{equation}
    with $J = \E\left[\frac{\partial l_t}{\partial
    \tau}(\tau_0)\frac{\partial l_t}{\partial \tau'}(\tau_0)\right]$.
\end{lem}

\begin{proof}
    For $\lambda\in \mathbb{R}^{p+q+4}$, we define
    $Y_{nt}=\lambda'\frac{\partial l_{nt}}{\partial\tau}(\tau_0)$, $Y_t
    =\lambda'\frac{\partial l_{t}}{\partial\tau}(\tau_0)$ and $z_{nt} =
    \frac{1}{\sqrt{n}}\left(Y_{nt}-\E Y_{nt}\right)$. We will apply
    the central limit theorem of Lindeberg to the array $(z_{nt})$ to prove this
    lemma. We obviously have $\E z_{nt} = 0$ and
    $\E\left[z_{nt}^2\right] = \frac{1}{n}\left(
    \E\left[Y_{nt}^2\right] -\E\left[Y_{nt}\right]^2\right)$. Now, in view of Lemma~\ref{ref: pseudo GARCH AN convergence d lnt en n}, we have,
    when $n$ tends to infinity
    \begin{equation*}
        \E\left[Y_{nt}^2\right] \rightarrow \E\left[Y_t^2\right]<+\infty,\
        \mbox{and}\ \E\left[Y_{nt}\right]^2 \rightarrow \E\left[Y_t\right]^2 = 0,
    \end{equation*}
    using the results of the proof of Lemma~\ref{lem: normalité
    asymptotique 1}. The Lindeberg condition remains to be proven. We
    have for any $\eps >0$,
    \begin{equation*}
        \sum_{t=1}^{n}\E\left[z_{nt}^2 \mathds{1}_{|z_{nt}|>\eps}\right] =
        \sum_{t=1}^{n} \frac{1}{n}\int_{|Y_{nt}-\E Y_{nt}|>\eps \sqrt{n}}
        |Y_{nt}-\E Y_{nt}|^2dP,
    \end{equation*}
    and $P\left[|Y_{nt}-\E Y_{nt}|>\eps \sqrt{n}\right]\rightarrow 0$,
    when $n\rightarrow +\infty$. Besides, we have
    $\underset{n\in\mathbb{N}}{\sup}\ \E \left|Y_{nt}-\E Y_{nt}\right|^2 <+\infty$. We can conclude and obtain the Lindeberg condition
    \begin{equation*}
        \sum_{t=1}^{n}\E\left[z_{nt}^2 \mathds{1}_{|z_{nt}|>\eps}\right] \rightarrow 0,\
        \mbox{when}\ n\rightarrow+\infty.
    \end{equation*}
    It remains to apply the Lindeberg central limit theorem and the
    Wold-Cramér theorem and we obtain~\eqref{eq: pseudo GARCH AN TCL}.
\end{proof}

\begin{lem}
    Under the assumptions of Theorem~\ref{theo: pseudo GARCH CAN}, we have
    \begin{equation}
        \label{eq: pseudo GARCH AN convergence d2I_n}
        \frac{\partial^2}{\partial\tau\partial\tau'}I_n(\tau_0) \rightarrow -J,\quad a.s.
    \end{equation}
\end{lem}

\begin{proof}
    Adapting~\eqref{eq:dérivée seconde theta theta} and
   ~\eqref{eq:dérivée seconde theta (alpha,beta)}, we obtain for
    $(i,j)\in\left\{1,\ldots,p+q+1\right\}$
    \begin{eqnarray*}
        \frac{\partial^2 l_{nt}}{\partial\theta_i\partial\theta_j} &=&
        \frac{1}{2} \left(\frac{1}{\sigma_{nt}^2}
        \frac{\partial^2\sigma_{nt}^2}{\partial\theta_i\theta_j}
        -\phi_{nt,i}\phi_{nt,j}\right)\left(1+\eta_{nt}
        \frac{\partial\log f}{\partial x}\right)\\
        &&-\frac{1}{4}\phi_{nt,i}\phi_{nt,j}\eta_{nt}
        \left(\frac{\partial\log f}{\partial x} +\eta_{nt}\frac{\partial^2\log f}{\partial x^2}\right),\\
        \frac{\partial^2 l_{nt}}{\partial\theta_i\partial\psi} &=& \frac{1}{2}
        \phi_{nt,i}\eta_t\frac{\partial^2\log f}{\partial x\partial\psi}\\
    \end{eqnarray*}

    Using (7.46) and (7.47) in \citet{francq2010garch} and the same
    reasoning as in Lemma~\ref{lem: pseudo GARCH, convergence des
    moments} (defining a truncated version of
    $\partial^2\sigma_{nt}^2/\partial\tau\partial\tau'$), we obtain for
    $(i,j)\in\left\{1,\ldots,p+q+1\right\}^2$
    \begin{equation*}
        \E\left[\frac{\partial^2 l_{nt}}{\partial\tau_i\partial\tau_j}(\tau_0)\right]
        \underset{n\rightarrow+\infty}{\longrightarrow}
        \E\left[\frac{\partial^2 l_{t}}{\partial\tau_i\partial\tau_j}(\tau_0)\right].
    \end{equation*}
    Then, as in Lemma~\ref{lem: pseudo LGN}, we obtain the result.
\end{proof}

\begin{lem}
    \label{lem: pseudo GARCH AN, dérivée 3}
    Under the assumptions of Theorem~\ref{theo: pseudo GARCH CAN},
    there exists a neighborhood $V(\tau_0)$ of  $\tau_0$ such that
    for $(i,j,k)\in\left\{1,\ldots,p+q+1\right\}^3$
    \begin{equation}
        \label{eq: pseudo GARCH AN, dérivée 3}
        \frac{1}{n}\sum_{t=1}^{n} \underset{\tau\in V(\tau_0)}{\sup}\
        \frac{\partial^3 l_{nt}}{\partial\tau_i\partial\tau_j\partial\tau_k}(\tau) \rightarrow
        \E\left[\underset{\tau\in V(\tau_0)}{\sup}\
        \frac{\partial^3 l_{t}}{\partial\tau_i\partial\tau_j\partial\tau_k}(\tau)
        \right],\quad a.s.
%        <+\infty
    \end{equation}
\end{lem}

\begin{proof}
    Using the results of Lemma~\ref{lem: pseudo GARCH, phi nt convergence en m} and of Lemma~\ref{lem:normalité asymptotique 2} the proof is straightforward.
\end{proof}

\begin{lem}
    \label{lem: last}
    Under the assumptions of Theorem~\ref{theo: pseudo GARCH CAN}, we
    have when $n\rightarrow +\infty$
    \begin{align}
        \left\|\frac{1}{\sqrt{n}}\sum_{t=1}^n\left\{\frac{\partial l_{nt}}{\partial\tau}(\tau_0)-
            \frac{\partial \tilde{l}_{nt}}{\partial\tau}(\tau_0)\right\}\right\|
            &\rightarrow 0
            \label{eq:conditions initiales sglouf 1},\\
        \underset{\tau\in\Gamma}{\sup}\left\|
            \frac1n\sum_{t=1}^n\left\{\frac{\partial^2 l_{nt}}{\partial\tau\partial\tau'}(\tau)-
            \frac{\partial^2 \tilde{l}_{n}t}{\partial\tau\tau'}(\tau)\right\}\right\|
            &\rightarrow 0.
            \label{eq:conditions initiales sglouf 2}
    \end{align}
\end{lem}

\begin{proof}
    This lemma can be easily proven using the same arguments as in Lemma
   ~\ref{lem:normalité asymptotique 3} and Lemma~\ref{lem: pseudo GARCH
    consistence conditions initiales}.
\end{proof}

\begin{proof}[Proof of Theorem~\ref{theo: pseudo GARCH CAN}]
    Using Lemmas~\ref{lem: pseudo GARCH, phi nt convergence en m}-\ref{lem:
    last} and the same method as in the proof of the asymptotic
    normality in the ML case, we obtain the result.
\end{proof}

\subsection{Proof of Theorem~\ref{thm: consistence de J_n}}

For $(i,j)\in\left\{1,\ldots, p+q+4\right\}^2$, a Taylor expansion yields
\begin{equation*}
    \frac{\partial^2 l_{nt}}{\partial\tau_i\partial\tau_j}(\tau_n) =
    \frac{\partial^2 l_{nt}}{\partial\tau_i\partial\tau_j}(\tau_0) +
    \frac{\partial^3 l_{nt}}{\partial\tau'\partial\tau_i\partial\tau_j}(\tilde{\tau})
    (\tau_n-\tau_0).
\end{equation*}
Using Lemma~\ref{lem:normalité asymptotique 2}, Lemma~\ref{lem:
pseudo GARCH AN, dérivée 3}, the equivalent of Lemma~\ref{ref:
pseudo GARCH AN convergence d lnt en n} for the second order
derivatives and the consistency of the estimator $\tau_n$, we obtain
\begin{equation*}
    \frac1n \sum_{t=1}^n \frac{\partial^2 l_{nt}}{\partial\tau_i\partial\tau_j}(\tau_n)
    \underset{n\rightarrow+\infty}{\longrightarrow} \E\left[
    \frac{\partial^2 l_{t}}{\partial\tau_i\partial\tau_j}(\tau_0)\right].
\end{equation*}

\newpage
\bibliographystyle{plainnat}
\bibliography{biblio}

%\bibliographystyle{plainnat}
%\bibliography{bibliostable}
\end{document}